\documentclass[a4paper,reqno]{amsart}

\usepackage{graphicx}
\usepackage{amsmath}
\usepackage{amssymb}
\input xy
\xyoption{all}

\theoremstyle{plain}
\newtheorem*{thmw}{Theorem}

\newtheorem{thm}{Theorem}[section]
\newtheorem{cor}[thm]{Corollary}
\newtheorem{lem}[thm]{Lemma}
\newtheorem{prop}[thm]{Proposition}

\theoremstyle{definition}
\newtheorem{defn}[thm]{Definition}
\newtheorem{exm}[thm]{Example}

\theoremstyle{remark}
\newtheorem{rem}[thm]{Remark}

\newdir{ >}{{}*!/-5pt/\dir{>}}

\renewcommand{\mod}{\operatorname{mod}\nolimits}
\newcommand{\Mod}{\operatorname{Mod}\nolimits}
\newcommand{\add}{\operatorname{add}\nolimits}
\newcommand{\Hom}{\operatorname{Hom}\nolimits}

\newcommand{\Ext}{\operatorname{Ext}\nolimits}
\newcommand{\Coker}{\operatorname{Coker}\nolimits}
\newcommand{\op}{\operatorname{op}\nolimits}

\newcommand{\M}{\mathcal M}

\newcommand{\B}{\mathcal B}
\newcommand{\uB}{\underline{\B}}

\newcommand{\U}{\mathcal U}
\newcommand{\V}{\mathcal V}
\newcommand{\W}{\mathcal W}
\newcommand{\h}{\mathcal H}
\newcommand{\s}{\mathcal S}
\newcommand{\T}{\mathcal T}

\newcommand{\svecv}[2]{\left(\begin{smallmatrix}
      #1 \\
      #2
    \end{smallmatrix}\right)}

\newcommand{\svech}[2]{\left(\begin{smallmatrix}
      #1 & #2
\end{smallmatrix}\right)}

\renewcommand{\emph}{\textit}
\renewcommand{\phi}{\varphi}

\begin{document}

\title{Hearts of twin cotorsion pairs on exact categories}
\author{Yu Liu}
\address{Graduate School of Mathematics \\ Nagoya University \\ 464-8602 Nagoya, Japan}
\email{d11005m@math.nagoya-u.ac.jp}

\begin{abstract}
In the papers of Nakaoka, he introduced the notion of hearts of (twin) cotorsion pairs on triangulated categories and showed that they have structures of (semi-) abelian categories. We study in this article a twin cotorsion pair $(\s,\T),(\U,\V)$ on an exact category $\B$ with enough projectives and injectives and introduce a notion of the heart. First we show that its heart is preabelian. Moreover we show the heart of a single cotorsion pair is abelian. These results are analog of Nakaoka's results in triangulated categories. We also consider special cases where the heart has nicer structure. By our results, the heart of a special twin cotorsion pair $(\s,\T),(\T,\V)$, is integral and almost abelian. Finally we show that the Gabriel-Zisman localisation of the heart at the class of regular morphisms is abelian, and moreover it is equivalent to the category of finitely presented modules over a stable subcategory of $\B$.

\end{abstract}

\keywords{exact category, preabelian category, abelian category, cotorsion pair, twin cotorsion pair, heart}

\maketitle

\tableofcontents

\section{Introduction}

The cotorsion pairs were first introduced by Salce in \cite{S}, and it has been deeply studied in the representation theory during these years, especially in tilting theory and Cohen-Macaulay modules \cite{A} (see \cite{HI} for more recent examples). Recently, the cotorsion pair are also studied in triangulated categories \cite{IY}, in particular, Nakaoka introduced the notion of hearts of cotorsion pairs and showed that the hearts are abelian categories \cite{N}. This is a generalization of the hearts of t-structure in triangulated categories \cite{BBD} and the quotient of triangulated categories by cluster tilting subcategories \cite{KZ}. Moreover, he generalized these results to a more general setting called twin cotorsion pair \cite{N1}.

The aim of this paper is to give similar results for cotorsion pairs on Quillen's exact categories, which plays an important role in representation theory \cite{K}. We consider a \emph{cotorsion pair} in an exact category (see for example \cite[{A.1}]{KS}), which is a pair $(\U,\V)$ of subcategories of an exact category $\B$ satisfying $\Ext^1_\B(\U,\V)=0$ (\emph{i.e.} $\Ext^1_\B(U,V)=0$, $\forall U\in \U$ and $\forall V\in \V$) and any $B\in \B$ admits two short exact sequences $V_B\rightarrowtail U_B\twoheadrightarrow B$ and $B\rightarrowtail V^B\twoheadrightarrow U^B$ where $V_B,V^B\in \V$ and $U_B,U^B\in \U$ (see Definition \ref{2} for more details). Let
\begin{align*}
\B^+:=\{B\in \B \text{ } | \text{ } U_B\in \V \}, \quad
\B^-:=\{B\in \B \text{ } | \text{ } V^B\in \U \}.
\end{align*}
We define the \emph{heart} of $(\U,\V)$ as the quotient category (see Definition \ref{6} for more details)
$$\underline \h:=(\B^+\cap \B^-)/(\U\cap \V).$$
An important class of exact categories is given by Frobenius categories, which gives most of important triangulated categories appearing in representation theory. Now we state our first main result, which is an analogue of  \cite[Theorem 6.4]{N}. We will prove it in Section 4.

\begin{thmw}\label{0.2}
Let $(\U,\V)$ be a cotorsion pair on an exact category $\B$ with enough projectives and injectives. Then $\underline \h$ is abelian.
\end{thmw}

Moreover, following Nakaoka, we consider pairs of cotorsion pairs $(\s,\T)$ and $(\U,\V)$ in $\B$ such that $\s\subseteq \U$, we also call such a pair a \emph{twin cotorsion pair} (see Definition \ref{ctp} for more details). The notion of hearts is generalized to such pairs (see Definition \ref{5} for more details), and our second main result is the following, which is an analogue of \cite[Theorem 5.4]{N1}.

\begin{thmw}\label{0.1}
Let $(\s,\T),(\U,\V)$ be a twin cotorsion pair on $\B$. Then $\underline \h$ is semi-abelian.
\end{thmw}

We will first prove $\underline \h$ is preabelian in Section 3 and then show the theorem above in Section 5.

The notion of semi-abelian category (see Definition \ref{semi}) was introduced by Rump \cite{R}, as a special class of preabelian categories. A especially nice class of semi-abelian categories is called integral (see Definition \ref{inter}, and see \cite[\S 2]{R} for examples). Our third main theorem gives sufficient conditions for hearts to be integral. We will show it in Section 6.

\begin{thmw}\label{0.4}
Let $(\s,\T),(\U,\V)$ be a twin cotorsion pair on $\B$ satisfying
$$\U\subseteq \s * \T, \text{ } {\mathcal P}\subseteq \W \quad \text{or} \quad \T\subseteq \U * \V, \text{ } {\mathcal I}\subseteq \W.$$
Then $\underline \h$ is integral.
\end{thmw}

Another nice class of semi-abelian categories is almost abelian categories. For example, any torsion class associated with a tilting module is almost abelian \cite{CF}. Our fourth main theorem gives sufficient conditions for hearts to be almost abelian. We will show it in Section 7.

\begin{thmw}\label{0.5}
Let $(\s,\T),(\U,\V)$ be a twin cotorsion pair on $\B$ satisfying
$$\U\subseteq  \T \text{ or }\T\subseteq \U.$$
Then $\underline \h$ is integral and almost abelian.
\end{thmw}

Finally, we consider a special twin cotorsion pair $(\s,\T),(\T,\V)$, note that this is an analog of TTF theory and recollement. Then we have the following theorem which gives a more explicit description of the heart and can be regarded as an analog of \cite[{Theorem 5.7}]{BM}. We will prove it in Section 9.

\begin{thmw}\label{0.0}
Let $(\s,\T),(\U,\V)$ be a twin cotorsion pair on $\B$ such that $\T=\U$. Let $R$ denote the class of regular morphisms in $\B/\T$ and $(\B/\T)_{R}$ denote the localisation of $\underline \h=\B/\T$ at $R$, then
$$(\B/\T)_{R} \simeq \mod (\Omega \s/\mathcal P)$$
where $\Omega \s$ consists of objects $\Omega S$ such that there exists a short exact sequence
$$\Omega S\rightarrowtail P\twoheadrightarrow S \text{ } (P\in \mathcal P, S\in \mathcal S).$$
\end{thmw}

In Section 2, we collect basic material on twin cotorsion pairs on $\B$. In Section 3-7 and 9, we prove our main results. In Section 8, we consider the cases when the heart of a twin cotorsion pair has enough projectives/injectives. In the last section we study some examples of twin cotorsion pairs.

\section{Preliminaries}

First we briefly review the important properties of exact categories. For more details, we refer to \cite{B}. Let $\mathcal A$ be an additive category, we call a pair of morphisms $(i,d)$ a \emph{weak short exact sequence} if $i$ is the kernel of $d$ and $d$ is the cokernel of $i$. Let $\mathcal E$ be a class of weak short exact sequences of $\mathcal A$, stable under isomorphisms, direct sums and direct summands. If a weak short exact sequence $(i,d)$ is in $\mathcal E$, we call it a \emph{short exact sequence} and denote it by
$$\xymatrix{X \;\ar@{>->}[r]^i &Y \ar@{->>}[r]^d &Z.}$$
We call $i$ an \emph{inflation} and $d$ a \emph{deflation}. The pair $(\mathcal A,\mathcal E)$ (or simply $\mathcal A$) is said to be an \emph{exact category} if it satisfies the following properties:

\begin{itemize}
\item[(a)] Identity morphisms are inflations and deflations.

\item[(b)] The composition of two inflations (resp. deflations) is an inflation (resp. deflation).

\item[(c)] If $\xymatrix{X \;\ar@{>->}[r]^i &Y \ar@{->>}[r]^d &Z}$ is a short exact sequence, for any morphisms $f : Z' \rightarrow Z$ and $g: X \rightarrow X'$, there are commutative diagrams
$$\xymatrix{
Y' \ar[d]_{f'} \ar@{->>}[r]^{d'} \ar@{}[dr]|{PB} & Z' \ar[d]^f\\
Y \ar@{->>}[r]_d & Z\\
} \quad \xymatrix{
X \ar[d]_{g} \ar@{}[dr]|{PO} \;\ar@{>->}[r]^i & Y \ar[d]^{g'}\\
X' \;\ar@{>->}[r]_{i'} & Y'\\
}$$
where $d'$ is a deflation and $i'$ is an inflation, the left square being a pull-back and the right being a push-out.
\end{itemize}

We introduce the following properties of exact category, the proofs of which can be find in \cite[{\S 2}]{B}:

\begin{prop}\label{PO}
Consider a commutative square
$$\xymatrix{
A \;\ar@{>->}[r]^{i} \ar[d]_f &B \ar[d]^{f'}\\
A' \;\ar@{>->}[r]_{i'} &B'
}
$$
in which $i$ and $i'$ are inflations. The following conditions are equivalent:
\begin{itemize}
\item[(a)] The square is a push-out.

\item[(b)] The sequence $\xymatrix{A \;\ar@{>->}[r]^-{\svecv{i}{-f}} &B\oplus A' \ar@{->>}[r]^-{\svech{f'}{i'}} &B'}$ is short exact.

\item[(c)] The square is both a push-out and a pull-back.

\item[(d)] The square is a part of a commutative diagram
$$\xymatrix{
A\; \ar@{>->}[r]^i \ar[d]_f &B \ar[d]^{f'} \ar@{->>}[r] &C \ar@{=}[d]\\
A'\; \ar@{>->}[r]_{i'} &B' \ar@{->>}[r] &C
}$$
with short exact rows.
\end{itemize}
\end{prop}

\begin{prop}\label{7}
\begin{itemize}
\item[(a)] If $\xymatrix{X \ar@{ >->}[r]^{i} &Y \ar@{->>}[r]^{d} &Z}$ and $\xymatrix{N \ar@{ >->}[r]^g &M\ar@{->>}[r]^{f} &Y}$ are two short exact sequences, then there is a commutative diagram of short exact sequences
$$\xymatrix{
N \ar@{ >->}[d] \ar@{=}[r] &N \ar@{ >->}[d]^g\\
Q \ar@{->>}[d] \;\ar@{>->}[r] &M \ar@{->>}[d]^f \ar@{->>}[r] &Z \ar@{=}[d]\\
X \;\ar@{>->}[r]_i &Y \ar@{->>}[r]_d &Z
}$$
where the lower-left square is both a push-out and a pull-back.

\item[(b)] If $\xymatrix{X \ar@{ >->}[r]^{i} &Y \ar@{->>}[r]^{d} &Z}$ and $\xymatrix{Y \ar@{ >->}[r]^{g} &K \ar@{->>}[r]^f &L}$ are two short exact sequences, then there is a commutative diagram of short exact sequences
$$\xymatrix{
X \ar@{=}[d] \ar@{ >->}[r]^{i} &Y \ar@{ >->}[d]^g \ar@{->>}[r]^{d} &Z \ar@{ >->}[d]\\
X \ar@{ >->}[r]  &K \ar@{->>}[r] \ar@{->>}[d]^f &R \ar@{->>}[d]\\
&L \ar@{=}[r] &L
}$$
where the upper-right square is both a push-out and a pull-back.
\end{itemize}
\end{prop}

Let $\mathcal A$ be an exact category, an object $P$ is called projective in $\mathcal A$ if for any deflation $f:X\rightarrow Y$ and any morphism $g:P\rightarrow Y$, there exists a morphism $h:P\rightarrow X$ such that $g=fh$. $\mathcal A$ is said to have enough projectives if for any object $X\in \mathcal A$, there is an object $P$ which is projective in $\mathcal A$ and a deflation $p:P\rightarrow X$. Injective objects and having enough injectives are defined dually.

Throughout this paper, let $\B$ be a Krull-Schmidt exact category with enough projectives and injectives. Let $\mathcal P$ (resp. $\mathcal I$) be the full subcategory of projectives (resp. injectives) of $\B$.

\begin{defn}\label{2}
Let $\U$ and $\V$ be full additive subcategories of $\B$ which are closed under direct summands. We call $(\U,\V)$ a \emph{cotorsion pair} if it satisfies the following conditions:
\begin{itemize}
\item[(a)] $\Ext^1_\B(\U,\V)=0$.

\item[(b)] For any object $B\in \B$, there exits two short exact sequences
\begin{align*}
V_B\rightarrowtail U_B\twoheadrightarrow B,\quad
B\rightarrowtail V^B\twoheadrightarrow U^B
\end{align*}
satisfying $U_B,U^B\in \U$ and $V_B,V^B\in \V$.
\end{itemize}
\end{defn}

By definition of a cotorsion pair, we can immediately conclude:
\begin{lem}\label{3}
Let $(\U,\V)$ be a cotorsion pair of $\B$, then
\begin{itemize}
\item[(a)] $B$ belongs to $\U$ if and only if $\Ext^1_\B(B,\V)=0$.

\item[(b)] $B$ belongs to $\V$ if and only if $\Ext^1_\B(\U,B)=0$.

\item[(c)] $\U$ and $\V$ are closed under extension.

\item[(d)] $\mathcal P \subseteq \U$ and $\mathcal I \subseteq \V$.

\end{itemize}
\end{lem}

\begin{defn}\label{ctp}
A pair of cotorsion pairs $(\s,\T)$, $(\U,\V)$ on $\B$ is called a \emph{twin cotorsion pair} if it satisfies:
$$\s\subseteq \U.$$
\end{defn}

By definition and Lemma \ref{3} this condition is equivalent to $\Ext^1_\B(\s,\V)=0$, and also to $\V\subseteq \T$.

\begin{rem}
\begin{itemize}
\item[(a)] We also regard a cotorsion pair $(\U,\V)$ as a degenerated case of a twin cotorsion pair $(\U,\V),(\U,\V)$.
\item[(b)] If $(\s,\T),(\U,\V)$ is a twin cotorsion pair on $\B$, then $(\V^{\op},\U^{\op}),(\T^{\op},\s^{\op})$ is a twin cotorsion pair on $\B^{\op}$.
\end{itemize}
\end{rem}

\begin{defn}\label{5}
For any twin cotorsion pair $(\s,\T),(\U,\V)$, put
$$\W:=\T\cap \U.$$
\begin{itemize}
\item[(a)] $\B^+$ is defined to be the full subcategory of $\B$, consisting of objects $B$ which admits a short exact sequence
$$V_B\rightarrowtail U_B\twoheadrightarrow B$$
where $U_B\in \W$ and $V_B\in \V$.

\item[(b)] $\B^-$ is defined to be the full subcategory of $\B$, consisting of objects $B$ which admits a short exact sequence
$$B\rightarrowtail T^B\twoheadrightarrow S^B$$
where $T^B\in \W$ and $S^B\in \s$.
\end{itemize}
\end{defn}

By this definition we get $\s \subseteq \U \subseteq \B^-$ and $\V\subseteq \T \subseteq \B^+$.

\begin{defn}\label{6}
Let $(\s,\T)$, $(\U,\V)$ be a twin cotorsion pair of $\B$, we denote the quotient of $\B$ by $\W$ as $\uB:=\B/\W$. For any morphism $f\in \Hom_\B(X,Y)$, we denote its image in $ \Hom_{\uB}(X,Y)$ by $\underline f$. And for any subcategory $\mathcal C$ of $\B$, we denote by $\underline {\mathcal C}$ the subcategory of $\uB$ consisting of the same objects as $\mathcal C$. Put
$$\h:=\B^+\cap\B^-.$$
Since $\h\supseteq \W$, we have an additive full quotient subcategory
$$\underline \h:=\h/\W$$
which we call the \emph{heart} of twin cotorsion pair $(\s,\T)$, $(\U,\V)$.\\
The heart of a cotorsion pair $(\U,\V)$ is defined to be the heart of twin cotorsion pair $(\U,\V)$, $(\U,\V)$.
\end{defn}

We prove some useful lemmas for a twin cotorsion pair $(\s,\T),(\U,\V)$ in the following:

\begin{lem}\label{5.0}
Let $(\s,\T),(\U,\V)$ be a twin cotorsion pair on $\B$, then
\begin{itemize}
\item[(a)]  $\B^-$ is closed under direct summands. Moreover, if $X\in \B^-$ admits a short exact sequence
$$X\rightarrowtail W\twoheadrightarrow U$$
where $W\in \W$ and $U\in \U$, then any direct summand $X_1$ of $X$ admits a short exact sequence
$$X_1\rightarrowtail W\twoheadrightarrow Y$$
where $Y\in \U$.
\item[(b)] $\B^+$ is closed under direct summands. Moreover, if $X\in \B^+$ admits a short exact sequence
$$V\rightarrowtail W'\twoheadrightarrow X$$
where $W\in \W$ and $V\in \V$, then any direct summand $X_2$ of $X$ admits a short exact sequence
$$Z\rightarrowtail W'\twoheadrightarrow X_2$$
where $Z\in \V$.
\end{itemize}
\end{lem}

\begin{proof}
We only show (a), (b) is by dual.\\
Suppose $X_1\oplus X_2$ admits a short exact sequence
$$\xymatrix{X_1\oplus X_2 \ar@{ >->}[rr]^-{\svech{x_1}{x_2}} &&W \ar@{->>}[r] &U}$$
where $U\in\U$ and $W\in \W$. Then $x_1:X_1\rightarrow W$ is also an inflation by the properties of exact category. Let $x_1$ admit a short exact sequence
$$\xymatrix{X_1 \ar@{ >->}[r]^{x_1} &W \ar@{->>}[r] &Y.}$$
For any morphism $f:X_1\rightarrow V_0$ where $V_0\in \V$, consider a morphism $\svech{f}{0}:X_1\oplus X_2\rightarrow V_0$. Since $\Ext^1_\B(U,V_0)=0$, ${\svech{x_1}{x_2}}$ is a left $\V$-approximation of $W$, there exists a morphism $g:W\rightarrow V_0$ such that $\svech{f}{0}=\svech{gx_1}{gx_2}$.
$$\xymatrix{
&X_1 \ar@{ >->}[rr]^{x_1} \ar@/_/[ddl]_f \ar[d]^{\svecv{1}{0}} &&W \ar@{->>}[r] \ar@{=}[d] &Y \ar[d]\\
&X_1\oplus X_2 \ar@{ >->}[rr]^-{\svech{x_1}{x_2}} \ar[dl]^-{\svech{f}{0}} &&W \ar@{->>}[r] \ar@{.>}@/^15pt/[dlll]^g &U\\
V_0
}
$$
Hence $\Hom_\B(x_1,V_0):\Hom_\B(W,V_0)\rightarrow \Hom_\B(X_1,V_0)$ is surjective. By the following exact sequence
$$\Hom_\B(W,V_0)\xrightarrow{\Hom_\B(x_1,V_0)} \Hom_\B(X_1,V_0)\xrightarrow{0} \Ext^1_\B(Y,V_0)\rightarrow \Ext^1_\B(W,V_0)=0$$
we have $\Ext^1_\B(Y,V_0)=0$, which implies $Y\in \U$.
\end{proof}

\begin{lem}\label{P1}
\begin{itemize}
\item[(a)] If $\xymatrix{A \ar@{ >->}[r]^{f} &B \ar@{->>}[r]^g &U}$ is a short exact sequence in $\B$ with $U\in \U$, then $A\in \B^-$ implies $B\in \B^-$.

\item[(b)] If $\xymatrix{A \ar@{ >->}[r]^{f} &B \ar@{->>}[r]^g &S}$ is a short exact sequence in $\B$ with $S\in \s$, then $B\in \B^-$ implies $A\in \B^-$.

\end{itemize}
\end{lem}

\begin{proof}
(b) Since $B\in \B^-$, by definition, there exists a short exact sequence
$$\xymatrix{B \ar@{ >->}[r]^{w^B} &W^B \ar@{->>}[r] &S^B.}$$
Take a push-out of $g$ and $w^B$, by Proposition \ref{7}, we get a commutative diagram of short exact sequences
$$\xymatrix{
A \ar@{=}[d] \ar@{ >->}[r]^f &B \ar@{ >->}[d]^{w^B} \ar@{->>}[r]^g &S \ar@{ >->}[d]\\
A \ar@{ >->}[r]  &{W^B} \ar@{->>}[r] \ar@{->>}[d] &X \ar@{->>}[d]\\
&{S^B} \ar@{=}[r] &{S^B}.
}$$
We thus get $X\in \s$ since $\s$ is closed under extension. This gives $A\in \B^-$.\\
(a) Since $A\in \B^-$, it admits a short exact sequence
$$\xymatrix{
A \ar@{ >->}[r]^{w^A} &{W^A} \ar@{->>}[r] &{S^A}.}$$
where $W^A\in \W$ and $S^A\in \s$. Since $\Ext^1_\B(\s,\T)=0$, $w^A$ is a left $\T$-approximation of $A$. Thus there exists a commutative diagram of two short exact sequences
$$\xymatrix{
A \ar@{ >->}[r] \ar[d]_f &{W^A} \ar@{->>}[r] \ar[d] &{S^A} \ar[d]\\
B \ar@{ >->}[r]_{t^B} &{T^B} \ar@{->>}[r] &{S^B}.}
$$
It suffices to show $T^B\in \U$.\\
Apply $\Ext^1_\B(-,\V)$ to the following commutative diagram
$$\xymatrix{
A \ar@{ >->}[r]^{f} \ar[d]_{w^A} &B \ar@{->>}[r] \ar[d]^{t^B} &U\\
W^A \ar[r] &T^B
}
$$
since $\Ext^1_\B(\U,\V)=0$, we obtain the following commutative diagram
$$\xymatrix{
&{\Ext^1_\B(T^B,\V)} \ar[d]^{\Ext^1_\B(t^B,\V)} \ar@{.>}@/_/[dl] \ar[r] &\Ext^1_\B(W^A,\V)=0 \ar[d]\\
0=\Ext^1_\B(U,\V) \ar[r] &{\Ext^1_\B(B,\V)} \ar[r]_{\Ext^1_\B(f,\V)} &{\Ext^1_\B(A,\V)}.}
$$
It follows that $\Ext^1_\B(t^B,\V)=0$. Then from the following exact sequence
$$0=\Ext^1_\B(S^B,\V)\rightarrow \Ext^1_\B(T^B,\V)\xrightarrow{\Ext^1_\B(t^B,\V)=0} \Ext^1_\B(B,\V)$$
we get that $\Ext^1_\B(T^B,\V)=0$, which implies that $T^B\in \U$. Thus $T^B\in \W$ and $B\in B^-$.
\end{proof}

Dually, the following holds.
\begin{lem}\label{P2}
\begin{itemize}
\item[(a)] If $\xymatrix{T \ar@{ >->}[r] &A \ar@{->>}[r]^{f} &B}$ is a short exact sequence in $\B$ with $T\in \T$, then $B\in \B^+$ implies $A\in \B^+$.

\item[(b)] If $\xymatrix{V \ar@{ >->}[r] &A \ar@{->>}[r]^{f} &B}$ is a short exact sequence in $\B$ with $V\in \V$, then $A\in \B^+$ implies $B\in \B^+$.

\end{itemize}
\end{lem}

Now we give a proposition which is similar with \cite[Proposition 1.10]{A} and useful in our article.

\begin{prop}\label{sp}
Let $\T$ be a subcategory of $\B$ satisfying
\begin{itemize}
\item[(a)] $\mathcal P\subseteq\T$.

\item[(b)] $\T$ is contravariantly finite.

\item[(c)] $\T$ is closed under extension.
\end{itemize}
Then we get a cotorsion pair $(\T,\V)$ where
$$\V=\{X\in \B \text{ }| \text{ }\Ext^1_\B(\T,X)=0 \}.$$

\end{prop}

\begin{proof}
For any object $B\in\B$, it admits a short exact sequence
$$\xymatrix{B \ar@{ >->}[r] &I \ar@{->>}[r]^f &X}$$
where $I\in \mathcal I$. By (a) and (b), we can take two short exact sequences
\begin{align*}
\xymatrix{V_X \ar@{ >->}[r] &T_X \ar@{->>}[r]^{t_X} &X,} \quad
\xymatrix{V_B \ar@{ >->}[r] &T_B \ar@{->>}[r]^{t_B} &B}
\end{align*}
where $t_X$ (resp. $t_B$) is a minimal right $\T$-approximation of $X$ (resp. $B$). Since $\T$ is closed under extension, by Wakamatsu's Lemma, we obtain $V_X\in \V$ (resp. $V_B\in \V$). Take a pull-back of $f$ and $t_X$, we get the following commutative diagram
$$\xymatrix{
&V_X \ar@{=}[r] \ar@{ >->}[d] &V_X \ar@{ >->}[d]\\
B \ar@{ >->}[r] \ar@{=}[d] &Y \ar@{->>}[r] \ar@{->>}[d] &T_X \ar@{->>}[d]^{t_X}\\
B \ar@{ >->}[r] &I \ar@{->>}[r]_f &X.
}
$$
Since $I,V\in \V$ and $\V$ is extension closed, we get $Y\in \V$. Thus $B$ admits two short exact sequence
\begin{align*}
V_B\rightarrowtail T_B\twoheadrightarrow B,\quad
B\rightarrowtail Y\twoheadrightarrow T_X
\end{align*}
satisfying $V_B,Y\in \V$ and $T_B,T_X\in \T$.
Hence by definition $(\T,\V)$ is a cotorsion pair.
\end{proof}

\section{\underline{$\mathcal {H}$}$\text{ }$is preabelian}

In this section, we fix a twin cotorsion pair $(\s,\T)$, $(\U,\V)$, we will show that the heart $\underline \h$ of a twin cotorsion pair is preabelian.

\begin{defn}\label{re}
For any $B\in\B$, define $B^+$ and $b^+:B\rightarrow B^+$ as follows:\\
Take two short exact sequences:
\begin{align*}
V_B\rightarrowtail U_B\twoheadrightarrow B,\quad
U_B\rightarrowtail T^U\twoheadrightarrow S^U
\end{align*}
where $U_B \in \U$, $\V_B\in \V$, $T^U\in \T$ and $S^U\in \s$. By Proposition \ref{7}, we get the following commutative diagram
\begin{equation}\label{F1}
$$\quad \quad \quad \quad \quad \quad \quad \quad \quad \quad \quad \quad \quad \quad \quad \quad \quad \quad\xymatrix{
V_B \ar@{=}[d] \ar@{ >->}[r] &U_B \ar@{ >->}[d]^u \ar@{->>}[r] &B \ar@{ >->}[d]^{b^+}\\
V_B \ar@{ >->}[r]  &T^U \ar@{->>}[r]^t \ar@{->>}[d] &B^+ \ar@{->>}[d]\\
&S^U \ar@{=}[r] &S^U
}$$
\end{equation}
where the upper-right square is both a push-out and a pull-back.
\end{defn}

We can easily get the following Lemma.

\begin{lem}\label{3.1}
By Definition \ref{re}, $B^+\in \B^+$. Moreover, if $B\in \B^-$, then $B^+\in \h$.
\end{lem}

\begin{proof}
Since $\U$ is closed under extension, we get $T^U\in\U\cap \T=\W$. Hence by definition $B^+\in \B^+$. If $B\in \B^-$, by Lemma \ref{P1}, $B^+$ also lies in $\B^-$. Thus $B^+\in \h$.
\end{proof}

We give an important property of $b^+$ in the following proposition.

\begin{prop}\label{8}
For any $B\in \B$ and $Y\in \B^+$, $\Hom_\B(b^+,Y):\Hom_\B(B^+,Y)\rightarrow \Hom_\B(B,Y)$ is surjective and $\Hom_{\uB}(\underline {b^+},Y):\Hom_{\uB}(B^+,Y)\rightarrow \Hom_{\uB}(B,Y)$ is bijective.
\end{prop}

\begin{proof}
Let $y\in \Hom_\B(B,Y)$ be any morphism. By definition, there exists a short exact sequence
$$\xymatrix{V_Y \ar@{ >->}[r] &W_Y \ar@{->>}[r]^{w_Y} &Y.}$$
Since $\Ext^1_\B(U_B,V_Y)=0$, $w_Y$ is a right $\U$-approximation of $Y$. Thus any $f \in \Hom_\B(U_B,Y)$ factors through $W_Y$.
$$\xymatrix{
&&U_B \ar[d]^f \ar@{.>}[dl]_g\\
V_Y \ar@{ >->}[r] &W_Y \ar@{->>}[r]_{w_Y} &Y
}
$$
As $\Ext^1_\B(\s,\T)=0$, $u$ is a left $\T$-approximation of $U_B$, we get the following commutative diagram:
$$\xymatrix{
U_B \ar[d]_g \ar@{ >->}[r]^u &T^U \ar@{.>}[dl] \ar@{->>}[r] &S^U\\
W_Y \ar[d]_{w_Y}\\
Y
}$$
which implies that $\Hom_\B(u,Y):\Hom_\B(T^U,Y)\rightarrow \Hom_\B(U_B,Y)$ is epimorphic. Hence when we apply $\Hom_\B(-,Y)$ to the diagram \eqref{F1}, we obtain the following exact sequence
$$\xymatrix{
\Hom_\B(B^+,Y) \ar[rr]^{\Hom_\B(b^+,Y)} &&{\Hom_\B(B,Y)} \ar[r] \ar@{.>}[d] &{\Ext^1_\B(S,Y)} \ar[r] \ar@{=}[d] &{\Ext^1_\B(B^+,Y)} \ar[d]\\
\Hom_\B(T^U,Y) \ar[rr]^{\Hom_\B(u,Y)} &&{\Hom_\B(U_B,Y)} \ar[r]^{0} &{\Ext^1_\B(S^U,Y)} \ar[r] &{\Ext^1_\B(T^U,Y)}
}$$
which implies that $\Hom_\B(b^+,Y)$ is an epimorphism. In particular, $\Hom_{\uB}(\underline {b^+},Y)$ is an epimorphism.\\
It remains to show that $\Hom_{\uB}(\underline {b^+},Y)$ is monomorphic. Suppose $q\in \Hom_\B(B^+,Y)$ satisfies $\underline {qb^+}=0$, it follows that $qb^+$ factors through $\W$. Since $w_Y$ is a right $\U$-approximation, there exists a morphism $a:B\rightarrow W_Y$ such that $w_Ya=qb^+$. Take a push-out of $b^+$ and $a$, we get the following commutative diagram of short exact sequences
$$\xymatrix{
B \ar@{ >->}[r]^{b^+} \ar[d]_{a} \ar@{}[dr]|{PO} &{B^+} \ar[d]^{c'} \ar@{->>}[r] &S^U \ar@{=}[d]\\
W_Y \ar@{ >->}[r]_c &Q \ar@{->>}[r] &S^U.}
$$
There exists a morphism $d:Q\rightarrow Y$ such that $dc=w_Y$ and $dc'=q$ by the definition of push-out. But $Q\in \U$ by Lemma \ref{3}, and $w_Y$ is a right $\U$-approximation, we have that $d$ factors through $W_Y$. Thus $q=dc'$ also factors through $W_Y$, and $\underline q=0$.
\end{proof}

We give an equivalent condition for a special case when $B^+=0$ in $\uB$.

\begin{lem}\label{eq}
For any $B\in \B$, the following are equivalent.
\begin{itemize}
\item[(a)] $B^+\in \W.$

\item[(b)] $B\in \U.$

\item[(c)] $\underline {b^+}=0$ in $\uB$.
\end{itemize}
\end{lem}

\begin{proof}
Consider the diagram \eqref{F1} in Definition \ref{re}. We first prove that (a) implies (b).\\
Suppose (b) holds. Since $B\in \U$, we get $B^+\in \U$. Thus $\Ext^1_\B(B^+,V_B)=0$, and then $t$ splits. Hence $B^+$ is a direct summand of $T^U\in \W$, which implies that $B^+\in \W$.\\
Obviously (a) implies (c), now it suffices to show that (c) implies (b).\\
Since $b^+$ factors through $\W$, and $t$ is a right $\U$-approximation of $B^+$, we get that $b^+$ factors through $t$. Hence by the definition of pull-back, the first row of diagram \eqref{F1} splits, which implies that $B\in \U$.
\end{proof}

Now we give a dual construction.

\begin{defn}\label{10}
For any object $B\in \B$, we define $b^-:B^- \rightarrow B$ as follows
Take the following two short exact sequences
\begin{align*}
B\rightarrowtail T^B \twoheadrightarrow S^B,\quad
V_T\rightarrowtail U_T \twoheadrightarrow T^B
\end{align*}
where $U_T\in \U$, $V_T\in \V$, $T^B\in \T$ and $S^B\in \s$. By Proposition \ref{7}, we get the following commutative diagram:
$$\xymatrix{
V_T \ar@{ >->}[d] \ar@{=}[r] &{V_T} \ar@{ >->}[d]\\
B^- \ar@{->>}[d]_{b^-} \ar@{ >->}[r] &{U_T} \ar@{->>}[d] \ar@{->>}[r] &{S^B} \ar@{=}[d]\\
B \ar@{ >->}[r] &{T^B} \ar@{->>}[r] &{S^B}.}
$$
\end{defn}

By duality, we get:

\begin{prop}
For any $B\in \B$, $B^-\in \B^-$ and $B\in \B^+$ implies $B^- \in \h$. For any $X\in \B^-$, $\Hom_\B(X,b^-):\Hom_\B(X,B^-)\rightarrow \Hom_\B(X,B)$ is surjective and $\Hom_{\uB}( X,\underline {b^-}):\Hom_{\uB}(X,B^-)\rightarrow \Hom_{\uB}(X,B)$ is bijective.
\end{prop}

\begin{defn}\label{MF}
For any morphism $f:A\rightarrow B$ with $A \in \B^-$, define $C_f$ and $c_f:B\rightarrow C_f$ as follows:\\
By definition, there exists a short exact sequence
$$\xymatrix{A \ar@{ >->}[r]^{w^A} &W^A \ar@{->>}[r] &S^A.}$$
Take a push-out of $f$ and $w^A$, we get the following commutative diagram of short exact sequences
\begin{equation}\label{F2}
$$\quad \quad \quad \quad \quad \quad \quad \quad \quad \quad \quad \quad \quad \quad \quad \quad \quad \quad \xymatrix{
A \ar@{ >->}[r]^{w^A} \ar[d]_{f} \ar@{}[dr]|{PO} &{W^A} \ar@{->>}[r] \ar[d] &{S^A} \ar@{=}[d]\\
B \ar@{ >->}[r]_{c_f} &{C_f} \ar@{->>}[r]_s &{S^A}.}
$$
\end{equation}
\end{defn}

By Lemma \ref{P1}, $B\in \B^-$ implies $C_f\in \B^-$.

Dually, we have the following:

\begin{defn}\label{13}
For any morphism $f:A\rightarrow B$ in $\B$ with $B\in \B^+$, define $K_f$ and $k_f:K_f\rightarrow A$ as follows:\\
By definition, there exists a short exact sequence
$$\xymatrix{V_B \ar@{ >->}[r] &{W_B} \ar@{->>}[r]^{w_B} &B.}$$
Take a pull-back of $f$ and $w_B$, we get the following commutative diagram of short exact sequences
\begin{equation}\label{F3}
$$\quad \quad \quad \quad \quad \quad \quad \quad \quad \quad \quad \quad \quad \quad \quad \quad \quad \quad \xymatrix{
V_B \ar@{ >->}[r] \ar@{=}[d] &{K_f} \ar@{}[dr]|{PB} \ar@{->>}[r]^{k_f} \ar[d] &A \ar[d]^f\\
V_B \ar@{ >->}[r] &{W_B} \ar@{->>}[r]_{w_B} &B}
$$
\end{equation}
\end{defn}

By Lemma \ref{P2}, $A\in B^+$ implies $K_f\in B^+$.

The following lemma gives an important property of $c_f$:

\begin{lem}\label{14}
Let $f:A\rightarrow B$ be any morphism in $\B$ with $A\in \B^-$, take the notation of Definition \ref{MF}, then $c_f:B\rightarrow C_f$ satisfies the following properties:\\
For any $C\in \B$ and any morphism $g\in \Hom_\B(B,C)$ satisfying $\underline {gf}=0$, there exists a morphism $c:C_f\rightarrow C$ such that $cc_f=g$.
$$\xymatrix{
A \ar[r]^f &B \ar[rr]^-g \ar[dr]_{c_f} &&C\\
&&C_f \ar@{.>}[ur]_c}
$$
Moreover if $C\in \B^+$, then $\underline c$ is unique in $\uB$. The dual statement also holds for $k_f$ in Definition \ref{13}.
\end{lem}

\begin{proof}
Since $\underline {gf}=0$, $gf$ factors through $\W$. As $\Ext^1_\B(S_A,W^A)=0$, $w^A$ is a left $\W$-approximation of $A$. Hence  there exists $b:W^A\rightarrow C$ such that $gf=bw^A$. Then by the definition of push-out, we get the following commutative diagram
$$\xymatrix{
A \ar[d]_f \ar[r]^{w^A} &W^A \ar[d] \ar@/^/[ddr]^b\\
B \ar[r]^{c_f} \ar@/_/[drr]_g &{C_f} \ar@{.>}[dr]^c\\
&&C.}
$$
Now assume that $C\in \B^+$ and there exists $c':C_f\rightarrow C$ such that $c'c_f=g$. Since $(c'-c)c_f=0$, there exists a morphism $d:S^A\rightarrow C$ such that $c'-c=ds$. As $C$ admits a short exact sequence
$$\xymatrix{V_C \ar@{ >->}[r] &{W_C} \ar@{->>}[r]^{w_C} &C}$$
and $w_C$ is a right $\U$-approximation of $C$, we obtain that there exists a morphism $e:S^A\rightarrow W_C$ such that $w_Ce=d$. Hence $c'-c$ factors through $W_C$, and $\underline c=\underline {c'}$.
\end{proof}

\begin{thm}\label{15}
For any twin cotorsion pair $(\s,\T),(\U,\V)$, its heart $\underline \h$ is preabelian.
\end{thm}

\begin{proof}
We only show the construction of the cokernel. For any $A,B\in \h$ and any morphism $f:A\rightarrow B$, by Definition \ref{MF}, since $A,B\in \B^-$, it follows $\underline {c_ff}=0$ and $C_f\in \B^-.$
By Proposition \ref{8}, there exists ${c_f}^+:C_f\rightarrow {C_f}^+$ where ${C_f}^+\in \h$ by Lemma \ref{3.1}.
We claim that $\underline {{c_f}^+c_f}:B\rightarrow {C_f}^+$ is the cokernel of $\underline f$.\\
Let $Q$ be any object in $\h$, and let $r:B\rightarrow Q$ be any morphism satisfying $\underline {rf}=0$, then by Lemma \ref{14} and Proposition \ref{8}, there exists a commutative diagram
$$\xymatrix{
&Q\\
A \ar[ur]^0 \ar[r]_{\underline f} &B \ar[u]^{\underline r} \ar[r]_{\underline {c_f}} &{C_f} \ar[r]_{\underline {{c_f}^+}} \ar@{.>}[ul]^{\underline a} &{{C_f}^+.} \ar@{.>}@/_/[ull]^{\underline b}}
$$
The uniqueness of $\underline b$ follows from Lemma \ref{14} and Proposition \ref{8}.
\end{proof}

\begin{cor}\label{eq3}
Let $f:A\rightarrow B$ be a morphism in $\h$, the the followings are equivalent:
\begin{itemize}
\item[(a)] $\underline f$ is epimorphic in $\underline \h$.

\item[(b)] ${C_f}^+\in \W.$

\item[(c)] $C_f\in \U.$
\end{itemize}
\end{cor}

\begin{proof}
The equivalence of (b) and (c) is given by Lemma \ref{eq}.\\
By Theorem \ref{15}, $\underline {{c_f}^+c_f}$ is the cokernel of $\underline f$ in $\underline \h$. The equivalence of (a) and (b) follows immediately by this argument.
\end{proof}

\section{Abelianess of the hearts of cotorsion pairs}

In this section we fix a cotorsion pair $(\U,\V)$. We will prove that the heart $\underline \h=\B^+\cap \B^-/\U\cap\V$ of a cotorsion pair is abelian.

\begin{lem}\label{17}
Let $A,B\in \h$, and let
$$\xymatrix{C \ar@{ >->}[r]^{g} &A \ar@{->>}[r]^{f} &B}$$
be a short exact sequence in $\B$. If $\underline f$ is epimorphic in $\underline \h$, then $C$ belongs to $\B^-$.
\end{lem}

\begin{proof}
As $\underline f$ is epimorphic in $\underline \h$, we get $C_f\in \U$ by Corollary \ref{eq3}. By Definition \ref{MF}, we get following commutative diagram
\begin{equation}\label{F4}
$$\quad \quad \quad \quad \quad \quad \quad \quad \quad \quad \quad \quad \quad \quad \quad \quad \quad \quad \quad \xymatrix{
C \ar@{ >->}[d]_g \ar@{=}[r] &C \ar@{ >->}[d]^h\\
A \ar@{ >->}[r]^{w^A} \ar[d]_{f} \ar@{}[dr]|{PO} &{W^A} \ar@{->>}[r] \ar[d] &{U^A} \ar@{=}[d]\\
B \ar@{ >->}[r]_{c_f} &{C_f} \ar@{->>}[r] &{U^A}.}
$$
\end{equation}
The middle column shows that $C\in \B^-$.
\end{proof}

We need the following lemma to prove our theorem.

\begin{lem}\label{epi}
\begin{itemize}
\item[(a)] Let $f:A\rightarrow B$ be a morphism in $\B$ with $B\in \B^+$, then there exists a deflation $\alpha=\svech{f}{-w_B}:A\oplus W_B\twoheadrightarrow B$ in $\B$ such that $\underline \alpha=\underline f$.

\item[(b)] Let $f:A\rightarrow B$ be a morphism in $\B$ with $A\in \B^-$, then there exists an inflation $\alpha=\svecv{f}{-w^A}:A\rightarrowtail B\oplus W^A$ in $\B$ such that $\underline {\alpha'}=\underline f$.
\end{itemize}
\end{lem}

\begin{proof}
We only show the first one, the second is dual.\\
As $B\in \B^+$, it admits a short exact sequence
$$\xymatrix{V_B \ar@{ >->}[r] &W_B \ar@{->>}[r]^{w_B} &B.}$$
Take a pull-back of $f$ and $w_B$, we get a commutative diagram
$$\xymatrix{
V_B \ar@{=}[d] \ar@{ >->}[r] &C \ar[d] \ar@{->>}[r] &A \ar[d]^f\\
V_B \ar@{ >->}[r] &W_B \ar@{->>}[r]_{w_B} &B.
}
$$
By dual of Proposition \ref{PO}, we get a short exact sequence
$$\xymatrix{C \ar@{ >->}[r] &A\oplus W_B \ar@{->>}[rr]^-{\alpha=\svech{f}{-w_B}} &&B}$$
and consequently $\alpha$ is a deflation and $\underline \alpha=\underline f$.
\end{proof}

\begin{thm}\label{18}
For any cotorsion pair $(\U,\V)$ on $\B$, its heart
$\underline \h$ is an abelian category.
\end{thm}

\begin{proof}
Since $\underline \h$ is preabelian, it remains to show the following:
\begin{itemize}
\item[(a)] If $\underline f$ is epimorphic in $\underline \h$, then $\underline f$ is a cokernel of some morphism in $\underline \h$.

\item[(b)] If $\underline f$ is monomorphic in $\underline \h$, then $\underline f$ is a kernel of some morphism in $\underline \h$.
\end{itemize}
We only show (a), since (b) is dual.\\
For any morphism $\underline f:A\rightarrow B$ which is epimorphic in $\underline \h$, by Lemma \ref{epi}, it is enough to consider the case that $f$ is a deflation.\\
Let $f$ admit a short exact sequence:
$$\xymatrix{C \ar@{ >->}[r]^{g} &A \ar@{->>}[r]^{f} &B.}$$
By Lemma \ref{17}, we have $C\in \B^-$. By Proposition \ref{8}, there exists
$$c^+:C\rightarrow C^+$$
where $C^+$ lies in $\h$ by Lemma \ref{3.1}. As $A\in \B^+$, there exists $a:C^+\rightarrow A$ such that $ac^+=g$.
$$\xymatrix{
C \ar[dr]_{c^+} \ar[rr]^g &&A\\
&{C^+.} \ar[ur]_a}
$$
Since $\underline {fac^+}=\underline {fg}=0$, we have $\underline {fa}=0$ by Proposition \ref{8}. We claim that $\underline f$ is the cokernel of $\underline a$.\\
Let $Q$ be any object in $\h$ and $r:A\rightarrow Q$ be any morphism. By Proposition \ref{8}, $\underline {rg}=0$ if and only if $\underline {ra}=0$.\\
So it is enough to show that any $\underline r$ satisfying $\underline {rg}=0$ factors through $\underline f$.\\
If $\underline {rg}=0$, $rg$ factors through $\W$. Consider the second column of diagram \eqref{F4}, since $h$ is a left $\V$-approximation of $C$, there exists a morphism $c:W^A\rightarrow Q$ such that $rg=ch$. Since $h=w^Ag$, we get that $(r-cw^A)g=0$. Thus $r-cw^A$ factors through $f$, which implies that $\underline r$ factors through $\underline f$.
\end{proof}

\section{\underline{$\mathcal {H}$}$\text{ }$is semi-abelian}

In the following sections, we fix a twin cotorsion pair $(\s,\T),(\U,\V)$.

\begin{defn}\label{semi}
A preabelian category $\mathcal A$ is called \emph{left semi-abelian} if in any pull-back diagram
$$\xymatrix{
A \ar[r]^{\alpha} \ar[d]_{\beta} &B \ar[d]^{\gamma}\\
C \ar[r]_{\delta} &D}
$$
in $\mathcal A$, $\alpha$ is an epimorphism whenever $\delta$ is a cokernel. \emph{Right semi-abelian} is defined dually. $\mathcal A$ is called \emph{semi-abelian} if it is both left and right semi-abelian. In this section we will prove that the heart $\underline \h$ of a twin cotorsion pair is semi-ableian.
\end{defn}

\begin{lem}\label{4.1}
If morphism $\beta \in \Hom_{\underline \h}(B,C)$ is a cokernel of a morphism $\underline f\in \Hom_{\underline \h}(A,B)$, then $B$ admits a short exact sequence
$$B\rightarrowtail C'\twoheadrightarrow S$$
where $C'\in \h$, $C\simeq C'$ in $\underline \h$ and $S\in \s$.
\end{lem}

\begin{proof}
Let $\beta$ be the cokernel of $\underline f:A\rightarrow B$. By Theorem \ref{15}, the cokernel of $\underline f$ is given by $\underline {{c_f}^+c_f}$. Therefore ${C_f}^+\simeq C$ in $\underline \h$. Consider diagram \eqref{F4} and the diagram which induces $(C_f)^+$ by Definition \ref{re}:
$$ \xymatrix{
V \ar@{ >->}[r] \ar@{=}[d] &U \ar@{->>}[r] \ar@{ >->}[d] &{C_f} \ar@{ >->}[d]^{{c_f}^+}\\
V \ar@{ >->}[r] &T' \ar@{->>}[r] \ar@{->>}[d] &{{C_f}^+} \ar@{->>}[d]\\
&S' \ar@{=}[r] &S'\\
}
$$
By Proposition \ref{7}, we obtain the following commutative diagram of short exact sequences
$$\xymatrix{
B \ar@{=}[d] \ar@{ >->}[r]^{c_f} &{C_f} \ar@{ >->}[d]_{{c_f}^+} \ar@{->>}[r] &{S^A} \ar@{ >->}[d]\\
B \ar@{ >->}[r] &{{C_f}^+} \ar@{->>}[r] \ar@{->>}[d] &Q \ar@{->>}[d]\\
&S' \ar@{=}[r] &S'.
}
$$
From the third column we get $Q\in \s$. Hence we get the required short exact sequence.
\end{proof}

\begin{prop}\label{eq2}
Let $\xymatrix{A \ar@{ >->}[r]^{f} &B \ar@{->>}[r]^{g} &C}$ be a short exact sequence in $\B$ with $f$ in $\h$. If $g$ factors through $\U$, then $\underline f$ is epimorphic in $\underline \h$.
\end{prop}

\begin{proof}
By Corollary \ref{eq3}, it suffices to show that $C_f \in \U$.\\
By definition of $c_f:B\rightarrow C_f$, there is a commutative diagram of short exact sequences
$$\xymatrix{
A \ar@{ >->}[r]^{w^A} \ar@{ >->}[d]_{f} \ar@{}[dr]|{PO} &{W^A} \ar@{->>}[r] \ar@{ >->}[d] &{S^A} \ar@{=}[d]\\
B \ar@{ >->}[r]_{c_f} \ar@{->>}[d]_g &{C_f} \ar@{->>}[r] \ar@{->>}[d] &{S^A}\\
C \ar@{=}[r] &C.}
$$
Since $\Ext^1_\B(\W,\V)=0$, we get the following commutative diagram of exact sequence
$$\xymatrix{
\Ext^1_\B(C,\V) \ar@{=}[d] \ar[rr] &&{\Ext^1_\B(C_f,\V)} \ar[d]^{\Ext^1_\B(c_f,\V)} \ar[rr] &&\Ext^1_\B(W^A,\V)=0 \ar[d]\\
\Ext^1_\B(C,\V) \ar[rr]_{\Ext^1_\B(g,\V)} &&{\Ext^1_\B(B,\V)} \ar[rr]_{\Ext^1_\B(f,\V)} &&{\Ext^1_\B(A,\V)}.}
$$
Then $\Ext^1_\B(c_f,\V)$ factors through $\Ext^1_\B(g,\V)$. We have $\Ext^1_\B(g,\V)=0$ since $g$ factors through $\U$, thus we get $\Ext^1_\B(c_f,\V)=0$. Then from the following exact sequence
$$0=\Ext^1_\B(S^A,\V)\rightarrow \Ext^1_\B(C_f,\V)\xrightarrow{\Ext^1_\B(c_f,\V)=0} \Ext^1_\B(B,\V)$$
we obtain that $\Ext^1_\B(C_f,\V)=0$, which implies $C_f\in \U$.\\
\end{proof}

\begin{lem}\label{4.2}
Suppose $X\in \B^-$ admits a short exact sequence
$$\xymatrix{X \ar@{ >->}[r]^{x} &B \ar@{->>}[r] &U}$$
where $B\in \h$ and $U\in \U$. Then the unique morphism $\underline b\in \Hom_{\underline \h}(X^+,B)$ given by Proposition \ref{8} which satisfies $\underline {bx^+}=\underline x$ is epimorphic.
\end{lem}

\begin{proof}
By Definition \ref{re}, there exists a short exact sequence
$$\xymatrix{X \ar@{ >->}[r]^{x^+} &X^+ \ar@{->>}[r] &S}$$
where $S\in \s$. By Proposition \ref{8}, there exits $b:X^+\rightarrow B$ such that $bx^+=x$. Since $X\in \B^-$, we obtain $X^+\in \h$ by Lemma \ref{3.1}. Hence $X^+$ admits a short exact sequence
$$\xymatrix{X^+ \ar@{ >->}[r]^a &W \ar@{->>}[r] &S'}$$
where $W\in \W$ and $S'\in \s$. Take a push-out of $a$ and $b$, we get the following commutative diagram
$$\xymatrix{
X^+ \ar@{ >->}[r]^a \ar[d]_b &W \ar@{->>}[r] \ar[d] &S' \ar@{=}[d]\\
B \ar@{ >->}[r] &C \ar@{->>}[r] &S'
}
$$
which induces a short exact sequence
$$\xymatrix{X^+ \ar@{ >->}[r]^-{\svecv{b}{-a}} &B\oplus W \ar@{->>}[r] &C}$$
by Proposition \ref{PO}. By Proposition \ref{7}, we obtain the following commutative diagram
$$\xymatrix{
X \ar@{ >->}[r]^{x^+} \ar@{=}[d] &X^+ \ar@{->>}[r] \ar@{ >->}[d]^a &S \ar@{ >->}[d]\\
X \ar@{ >->}[r]_c &W \ar@{->>}[r] \ar@{ ->>}[d] &Q \ar@{->>}[d]\\
&S' \ar@{=}[r] &S'.
}
$$
Take a push-out of $x$ and $c$
$$\xymatrix{
X \ar@{ >->}[r]^c \ar@{ >->}[d]_x &W \ar@{->>}[r] \ar@{ >->}[d] &Q \ar@{=}[d]\\
B \ar@{ >->}[r] \ar@{->>}[d] &C' \ar@{->>}[r] \ar@{->>}[d] &Q\\
U \ar@{=}[r] &U
}
$$
from the second column we obtain that $C'\in \U$ and we get the following short exact sequence
$$\xymatrix{X \ar@{ >->}[r]^-{\svecv{x}{-c}} &B\oplus W \ar@{->>}[r] &C'}$$
by Proposition \ref{PO}. Thus we get the following commutative diagram
$$\xymatrix{
&X \ar@{ >->}[dr]^-{\svecv{x}{-c}} \ar@{}[d]|{\circlearrowright} \ar[dl]_{x^+}\\
X^+ \ar@{ >->}[rr]_-{\svecv{b}{-a}} &&B\oplus W \ar@{->>}[rr] \ar@{->>}[dr] &&C.\\
&&&C' \ar@{.>}[ur]
}
$$
Hence by Proposition \ref{eq2}, $\underline b$ is epimorphic.
\end{proof}

We introduce the following lemma which is an analogue of \cite[Lemma 5.3]{N1}.

\begin{lem}\label{4.4}
Let
$$\xymatrix{
A \ar[r]^{\alpha} \ar[d]_{\beta} &B \ar[d]^{\gamma}\\
C \ar[r]_{\delta} &D}
$$
be a pull-back diagram in $\underline \h$. If there exists an object $X\in \B^-$ and morphisms $x_B:X\rightarrow B$, $x_C:X\rightarrow C$ which satisfy the following conditions, then $\alpha$ is epimorphic in $\underline \h$.
\begin{itemize}
\item[(a)] The following diagram is commutative.
$$\xymatrix{
X \ar[r]^{\underline {x_B}} \ar[d]_{\underline {x_C}} &B \ar[d]^{\gamma}\\
C \ar[r]_{\delta} &D
}
$$

\item[(b)] There exists a short exact sequence $\xymatrix{X \ar@{ >->}[r]^{x_B} &B \ar@{->>}[r] &U}$ with $U\in \U$.
\end{itemize}
\end{lem}

\begin{proof}
Take $x^+:X\rightarrow X^+$ as in Definition \ref{re}. Then by Proposition \ref{8}, there exist $f_B:X^+\rightarrow B$ and $f_C:X^+\rightarrow C$ such that $\underline {f_Bx^+}=\underline {x_B}$ and $\underline {f_Cx^+}=\underline {x_C}$. By Lemma \ref{4.2}, $\underline {f_B}$ is epimorphic in $\underline \h$. As $\gamma\underline {x_B}=\delta\underline {x_C}$, we get $\gamma\underline {f_Bx^+}=\delta\underline {f_Cx^+}$, it follows by Proposition \ref{8} that $\gamma\underline {f_B}=\delta\underline {f_C}$. By the definition of pull-back, there exists a morphism $\eta:X^+\rightarrow A$ in $\underline \h$ which makes the following diagram commute.
$$\xymatrix{
X^+ \ar@/^/[drr]^{\underline {f_B}} \ar@/_/[ddr]_{\underline {f_C}} \ar@{.>}[dr]^{\eta}\\
&A \ar[r]^{\alpha} \ar[d]^{\beta} &B \ar[d]^{\gamma}\\
&C \ar[r]_{\delta} &D}
$$
Since $\underline {f_B}$ is epimorphic, we obtain that $\alpha$ is also epimorphic.
\end{proof}

\begin{thm}\label{4.3}
For any twin cotorsion pair $(\s,\T),(\U,\V)$, its heart $\underline \h$ is semi-abelian.
\end{thm}

\begin{proof}
By duality, we only show $\underline \h$ is left semi-abelian. Assume we are given a pull-back diagram
$$\xymatrix{
A \ar[r]^{\alpha} \ar[d]_{\beta} &B \ar[d]^{\gamma}\\
C \ar[r]_{\delta} &D}
$$
in $\underline \h$ where $\delta$ is a cokernel. It suffices to show that $\alpha$ becomes epimorphic.\\
By Lemma \ref{4.1}, replacing $D$ by an isomorphic one if necessary, we can assume that there exists an inflation $d:C\rightarrowtail D$ satisfying $\delta=\underline d$, which admits a short exact sequence
$$\xymatrix{C \ar@{ >->}[r]^{d} &D \ar@{->>}[r] &S}$$
where $S\in \s$. As $D\in \B^+$, by Lemma \ref{epi} we can also assume that there exists an deflation $c:B\twoheadrightarrow D$ such that $\gamma=\underline c$. By Proposition \ref{7}, we get the following commutative diagram of short exact sequences
$$\xymatrix{
X \ar@{ >->}[r]^{x_B} \ar@{->>}[d]_{x_C} &B \ar@{->>}[d]^c \ar@{->>}[r] &S \ar@{=}[d]\\
C \ar@{ >->}[r]_d &D \ar@{->>}[r] &S.}
$$
it follows by Lemma \ref{P1} that $X\in \B^-$. Hence by Lemma \ref{4.4} $\alpha$ is epimorphic in $\underline \h$.
\end{proof}

\section{The case where \underline{$\mathcal {H}$}$\text{ }$becomes integral}

\begin{defn}\label{inter}
A preabelian category $\mathcal A$ is called \emph{left integral} if in any pull-back diagram
$$\xymatrix{
A \ar[r]^{\alpha} \ar[d]_{\beta} &B \ar[d]^{\gamma}\\
C \ar[r]_{\delta} &D}
$$
in $\mathcal A$, $\alpha$ is an epimorphism whenever $\delta$ is an epimorphic. \emph{Right integral} is defined dually. $\mathcal A$ is called \emph{integral} if it is both left and right integral.
\end{defn}

In this section we give a sufficient condition where the heart $\underline \h$ becomes integral.\\
Let $\mathcal C$ be a subcategory of $\B$, denote by $\Omega \mathcal C$ (resp. $\Omega^- \mathcal C$) the subcateogy of $\B$ consisting of objects $\Omega C$ (resp. $\Omega^- C$) such that there exists a short exact sequence
\begin{align*}
\Omega C\rightarrowtail P_C\twoheadrightarrow C \text{ } (P\in \mathcal P, C\in \mathcal C)\\
(\text{resp. } C\rightarrowtail I^C\twoheadrightarrow \Omega^- C \text{ } (I\in \mathcal I, C\in \mathcal C)).
\end{align*}
By definition we get $\mathcal P\subseteq \Omega \mathcal C$ and $\mathcal I\subseteq \Omega^- \mathcal C$. By Lemma \ref{5.0} we get that for any cotorsion pair $(\U,\V)$ on $\B$, $\Omega \U$ and $\Omega^- \V$ are closed under direct summands.

Let $\B_1$ $\B_2$ be two subcategories of $\B$, recall that $\B_1*\B_2$ is subcategory of $\B$ consisting of objects $X$ such that there exists a short exact sequence
$$B_1 \rightarrowtail X\twoheadrightarrow \B_2$$
where $B_1\in\B_1$ and $B_2\in\B_2$.

\begin{thm}\label{5.1}
If a twin cotorsion pair $(\s,\T),(\U,\V)$ satisfies
$$\U\subseteq \s * \T, \text{ } {\mathcal P}\subseteq \W \quad \text{or} \quad \T\subseteq \U * \V, \text{ } {\mathcal I}\subseteq \W$$
then $\underline \h$ becomes integral.
\end{thm}

\begin{proof}
According to \cite[{Proposition 6}]{R}, a semi-abelian category is left integral if and only if it is right integral.
By duality, it suffices to show that $\U\subseteq \s * \T, {\mathcal P}\subseteq \W$ implies that $\underline \h$ is left integral.
Assume we are given a pull-back diagram
$$\xymatrix{
A \ar[r]^{\alpha} \ar[d]_{\beta} &B \ar[d]^{\gamma}\\
C \ar[r]_{\delta} &D}
$$
in $\underline \h$ where $\delta$ is an epimorphism. It is sufficient to show that $\alpha$ is epimorphic.\\
Let $d:C\rightarrow D$  and $c:B\rightarrow D$ be morphisms satisfying $\delta=\underline d$ and $\gamma=\underline c$. Since $\delta$ is epimorphic, if we take $c_d:D\rightarrow C_d$ as in Definition \ref{MF}
$$\xymatrix{
C \ar@{ >->}[r]^{w^C} \ar[d]_d  \ar@{}[dr]|{PO} &W^C \ar@{->>}[r] \ar[d] &S^C \ar@{=}[d]\\
D \ar@{ >->}[r]_{c_d} &C_d \ar@{->>}[r]_r &S^C}
$$
then $C_d\in \U$ by Corollary \ref{eq3}.
By assumption $\U\subseteq \s * \T$, $C_d$ admits a short exact sequence
$$\xymatrix{S_0 \ar@{ >->}[r]^{s_0} &C_d \ar@{->>}[r]^{t_0} &T_0}$$
with $S_0\in \s$, $T_0\in \T$.
Since $B\in \B^-$ admits a short exact sequence
$$B\rightarrowtail W^B\twoheadrightarrow S^B$$
and $S^B$ admits a short exact sequence
$$\xymatrix{\Omega S^B \ar[r]^p &P_{S^B} \ar@{->>}[r]^s  &S^B}$$
there exists a commutative diagram
\begin{equation}\label{F5}
$$\quad \quad \quad \quad \quad \quad \quad \quad \quad \quad \quad \quad \quad \quad \quad \quad \xymatrix{
\Omega S^B \ar[d]_{s_B} \ar@{ >->}[r]^p &P_{S^B} \ar@{->>}[r]^s \ar[d] &S^B \ar@{=}[d]\\
B \ar@{ >->}[r] &W^B \ar@{->>}[r] &S^B.
}$$
\end{equation}
As $\Ext^1_\B(S_B,T_0)=0$, $p$ is a left $\T$-approximation of $\Omega S^B$. Therefore there exists a morphism $f:P_{S^B}\rightarrow T_0$ such that $t_0c_dcs_B=fp$. As $P_{S^B}\in \mathcal P$, there is a morphism $h:P_{S^B}\rightarrow C_d$ such that $f=t_0h$. Since $t_0(c_dcs_B-hp)=0$, there exists a morphism $g:\Omega S_B\rightarrow S_0$ such that $c_dcs_B-hp=s_0g$. Then we get the following diagram
$$\xymatrix{
\Omega S^B \ar[d]_{s_B} \ar@{.>}[dr]^{g} \ar@{ >->}[rr]^p &&{P_{S^B}} \ar@{.>}@/^10pt/[dddl]^f \ar@{.>}@/^/[ddl]_h \ar@{->>}[r]^s &S^B\\
B \ar[d]_c &S_0 \ar@{ >->}[d]_{s_0}\\
D \ar[r]^{c_d} &C_d \ar@{->>}[d]_{t_0}\\
&T_0.}
$$
Take a push-out of $p$ and $g$, we get the following commutative diagram
$$\xymatrix{
\Omega S^B \ar@{ >->}[r]^{p} \ar[d]_{g} &{P_{S^B}} \ar@{->>}[r]^s \ar[d] &S^B \ar@{=}[d]\\
S_0 \ar@{ >->}[r] &Q \ar@{->>}[r] &S^B}
$$
and a short exact sequence
$$\xymatrix{\Omega S^B \ar@{ >->}[r]^-{\svecv{p}{-g}} &P_{S^B}\oplus S_0 \ar@{->>}[r] &Q}$$
by Proposition \ref{PO} where $Q\in \s$.
As $Q$ admits a short exact sequence
$$\xymatrix{\Omega Q \ar@{ >->}[r]^{k_Q} &P_Q \ar@{->>}[r]^{l_Q} &Q}$$
we get the following commutative diagram of short exact sequences
\begin{equation}\label{F6}
$$\quad \quad \quad \quad \quad \quad \quad \quad \quad \quad \quad \quad \quad \quad \quad \quad \xymatrix{
\Omega Q \ar[d]_{q_B} \ar@{ >->}[r]^{k_Q} &P_Q \ar@{->>}[r]^{l_Q} \ar[d] &Q \ar@{=}[d]\\
\Omega S^B \ar@{ >->}[r]_-{\svecv{p}{-g}} &{P_{S^B}\oplus S_0} \ar@{->>}[r] &Q.
}
$$
\end{equation}
Since $c_dcs_B=hp+s_0g$, we obtain the following commutative diagram of short exact sequences.
$$\xymatrix{
\Omega S^B \ar[d]_{cs_B} \ar@{ >->}[r]^-{\svecv{p}{-g}} &{P_{S^B}\oplus S_0} \ar@{->>}[r] \ar[d]^-{\svech{h}{-s_0}} &Q \ar[d]\\
D \ar@{ >->}[r]_{c_d} &C_d \ar@{->>}[r]_r &S^C}
$$
Thus we get the following commutative diagram
$$\xymatrix{
\Omega Q \ar[d]_{cs_Bq_B} \ar@{ >->}[r]^{k_Q} &P_Q \ar@{->>}[r]^{l_Q} \ar[d]^{n_Q} &Q \ar[d]\\
D \ar@{ >->}[r]_{c_d} &C_d \ar@{->>}[r]_r &S^C.}
$$
As $\mathcal P\subseteq \W$, we conclude that $\Omega Q\in \B^-$. Since $S^C$ admits a short exact sequence
$$\xymatrix{\Omega S^C \ar@{ >->}[r]^{k_{S^C}} &{P_{S^C}} \ar@{->>}[r]^{l_{S^C}} &S^C}$$
where $P_{S^C}\in \mathcal P$, hence we get the following commutative diagram of short exact sequence
$$\xymatrix{
\Omega S^C \ar[d]_{q_C} \ar@{ >->}[r]^{k_{S^C}} &{P_{S^C}} \ar@{->>}[r]^{l_{S^C}} \ar[d] &S^C \ar@{=}[d]\\
C \ar@{ >->}[r]^{w^C} \ar[d]_d  &W^C \ar@{->>}[r] \ar[d] &S^C \ar@{=}[d]\\
D \ar@{ >->}[r]_{c_d} &C_d \ar@{->>}[r]_r &S^C}
$$
which induces the following diagram
$$\xymatrix{
\Omega S^C \ar[d]_{dq_C} \ar@{ >->}[r]^{k_{S^C}} &{P_{S^C}} \ar@{->>}[r]^{l_{S^C}} \ar[d]^{n_{S^C}} &S^C \ar@{=}[d]\\
D \ar@{ >->}[r]_{c_d} &C_d \ar@{->>}[r]_r &S^C.}
$$
As $P_Q$ is projective, there exists a morphism $t:P_Q\rightarrow P_{S_C}$ such that $l_{S_C}t=rn_{Q}$.
$$\xymatrix{
&P_Q \ar[r]^{n_Q} \ar@{.>}[d]_t &C_d \ar[d]^r\\
\Omega S^C \ar@{ >->}[r]_{k_{S^C}}&P_{S^C} \ar@{->>}[r]_{l_{S^C}} &S^C}
$$
Now it follows that $l_{S^C}tk_Q=rn_{Q}k_Q=rc_dcs_Bq_B=0$, thus there exists a morphism $x:\Omega Q \rightarrow \Omega S_C$ such that $k_{S^C}x=tk_Q$.
$$\xymatrix{
\Omega Q \ar[r]^{k_Q} \ar@{.>}[d]_x &P_Q \ar[d]^t\\
\Omega S^C \ar@{ >->}[r]_{k_{S_C}} &P_{S^C} \ar@{->>}[r]_{l_{S^C}} &S^C.}
$$
As $rn_{S^C}t=l_{S^C}t=rn_Q$, there exists a morphism $y:P_Q\rightarrow D$ such that $n_{S^C}t-n_Q=c_dy$. Therefore
$$c_ddq_Cx=n_{S^C}k_{S^C}x=n_{S^C}tk_Q=(c_dy+n_Q)k_Q=c_d(yk_Q+cs_Bq_B).$$
Then $dq_Cx=yk_Q+cs_Bq_B$, since $c_d$ is monomorphic. Hence there exists a commutative diagram in $\uB$
$$\xymatrix{
\Omega Q \ar[r]^-{\underline {s_Bq_B}} \ar[d]_-{\underline {q_Cx}} &B \ar[d]^{\underline c}\\
C \ar[r]_{\underline d} &D.}
$$
By Proposition \ref{PO}, we get the following short exact sequences from \eqref{F5} and \eqref{F6}:
\begin{align*}
\xymatrix{\Omega Q \ar@{ >->}[r]^-{\svecv{q_B}{-k_Q}} &\Omega S_B\oplus P_Q \ar@{->>}[r] &P_{S_B}\oplus S_0,}\quad
\xymatrix{\Omega S^B \ar@{ >->}[r]^-{\svecv{s_B}{-p}} &B\oplus P_{S^B} \ar@{->>}[r] &W^B.}
\end{align*}
Then by Proposition \ref{7}, we get the following commutative diagram of short exact sequences
$$\xymatrix{
\Omega Q \ar@{ >->}[r]^-{\svecv{q_B}{-k_Q}} \ar@{=}[d] &{\Omega S^B\oplus P_Q} \ar@{->>}[r] \ar@{ >->}[d]^{\left(\begin{smallmatrix}
      s_B &0\\
      -p &0\\
      0 &1
    \end{smallmatrix}\right)} &{P_{S^B}\oplus S_0} \ar@{ >->}[d]\\
\Omega Q \ar@{ >->}[r]_-{\eta} &{B\oplus P_{S^B}\oplus P_Q} \ar@{->>}[r] \ar@{->>}[d] &M \ar@{->>}[d]\\
&W^B \ar@{=}[r] &W^B
}$$
where $\underline \eta=\underline {s_Bq_B}$. From the third column we get that $M\in \U$. By Lemma \ref{4.4}, we obtain that $\alpha$ is epimorphic.
\end{proof}

\section{The case where \underline{$\mathcal {H}$}$\text{ }$becomes almost abelian}

\begin{defn}\label{almost}
A preabelian category $\mathcal A$ is called \emph{left almost abelian} if in any pull-back diagram
$$\xymatrix{
A \ar[r]^{\alpha} \ar[d]_{\beta} &B \ar[d]^{\gamma}\\
C \ar[r]_{\delta} &D}
$$
in $\mathcal A$, $\alpha$ is a cokernel whenever $\delta$ is a cokernel. \emph{Right almost abelian} is defined dually. $\mathcal A$ is called \emph{almost abelian} if it is both left and right almost abelian.
\end{defn}

In this section we give a sufficient condition when $\underline \h$ becomes almost abelian.

We need the following proposition to show our result.

\begin{prop}\cite[{Proposition 2}]{R}\label{r}
Let $A\xrightarrow{f} B\xrightarrow{g} C$ be morphisms in a right (resp. left) semi-abelian category. If $f$ and $g$ are (co-)kernels, then $gf$ is a (co-)kernel. If $gf$ is a (co-)kernel, then $f$ (resp. $g$) is a (co-)kernel.
\end{prop}

Use this proposition, we can prove the following lemma, which is an analogue of Lemma \ref{4.4}.

\begin{lem}\label{8.1}
Let
$$\xymatrix{
A \ar[r]^{\alpha} \ar[d]_{\beta} &B \ar[d]^{\gamma}\\
C \ar[r]_{\delta} &D}
$$
be a pull-back diagram in $\underline \h$. Let $X\in \B^-$ and $x_B:X\rightarrow B$, $x_C:X\rightarrow C$ be morphisms which satisfy that $\underline {x_B}$ is a cokernel in the following commutative diagram
$$\xymatrix{
X \ar[r]^{\underline {x_B}} \ar[d]_{\underline {x_C}} &B \ar[d]^{\gamma}\\
C \ar[r]_{\delta} &D.
}
$$
Then if $\U\subseteq \T$, we obtain $\alpha$ is a cokernel in $\underline \h$.
\end{lem}

\begin{proof}
Since $\U\subseteq \T$, we get $\h=\B^-$. Take $x^+:X\rightarrow X^+$ as in Definition \ref{re}. Then by Proposition \ref{8}, there exist $f_B:X^+\rightarrow B$ and $f_C:X^+\rightarrow C$ such that $\underline {f_Bx^+}=\underline {x_B}$ and $\underline {f_Cx^+}=\underline {x_C}$. Since $\underline {x_B}$ is a cokernel, by Proposition \ref{r}, $\underline {f_B}$ is also a cokernel in $\underline \h$. As $\gamma\underline {x_B}=\delta\underline {x_C}$, it follows by Proposition \ref{8} that $\gamma\underline {f_B}=\delta\underline {f_C}$. By the definition of pull-back, there exists a morphism $\eta:X^+\rightarrow A$ in $\underline \h$ which makes the following diagram commute.
$$\xymatrix{
X^+ \ar@/^/[drr]^{\underline {f_B}} \ar@/_/[ddr]_{\underline {f_C}} \ar@{.>}[dr]^{\eta}\\
&A \ar[r]^{\alpha} \ar[d]^{\beta} &B \ar[d]^{\gamma}\\
&C \ar[r]_{\delta} &D}
$$
Since $\underline {f_B}$ is a cokernel, we obtain that $\alpha$ is also a cokernel by Proposition \ref{r}.
\end{proof}

\begin{thm}\label{al}
Let $(\s,\T),(\U,\V)$ be a twin cotorsion pair on $\B$ satisfying
$$\U\subseteq \T \text{ or }\T\subseteq \U$$
then $\underline \h$ is almost abelian.
\end{thm}

\begin{proof}
By \cite[{Proposition 3}]{R}, a semi-abelian category is left almost abelian if and only if it is right almost abelian. By duality, it is enough to show that $\U\subseteq \T$ implies $\underline \h$ is left almost abelian.\\
Assume we are given a pull-back diagram
$$\xymatrix{
A \ar[r]^{\alpha} \ar[d]_{\beta} &B \ar[d]^{\gamma}\\
C \ar[r]_{\delta} &D}
$$
in $\underline \h$ where $\delta$ is a cokernel. It suffices to show that $\alpha$ becomes a cokernel.\\
Repeat the same argument as in Theorem \ref{4.3}, we get the following diagram
$$\xymatrix{
X \ar@{ >->}[r]^{x_B} \ar@{.>}[d]_{x_C} &B \ar@{.>}[d]^c \ar@{->>}[r] &S \ar@{=}[d]\\
C \ar@{ >->}[r]_d &D \ar@{->>}[r] &S}
$$
where $X\in \B^-$, $\underline d=\delta$ and $\underline c=\gamma$. According to Lemma \ref{8.1}, it suffices to show that $\underline {x_B}$ is a cokernel in $\underline \h$.\\
By Definition \ref{13} and Proposition \ref{7}, we get the following commutative diagram
$$\xymatrix{
V_B \ar@{ >->}[r] \ar@{=}[d] &K_{x_B} \ar@{ >->}[d]^a \ar@{->>}[r]^{k_{x_B}} &X \ar@{ >->}[d]^{x_B}\\
V_B \ar@{ >->}[r] &W_B \ar@{->>}[r] \ar@{->>}[d] &B \ar@{->>}[d]\\
&S \ar@{=}[r] &S.
}
$$
It follows that $K_{x_B}\in \B^-=\h$ and $\underline {k_{x_B}x_B}=0$. Now let $r:X\rightarrow Q$ be any morphism in $\h$ such that $\underline {rk_{x_B}}=0$, then $rk_{x_B}$ factors through $\W$. Since $\Ext^1_\B(\s,\T)=0$, $a$ is a left $\T$-approximation of $K_{x_B}$, thus there exists a morphism $b:W_B\rightarrow Q$ such that $ab=rk_B$. By the definition of push-out, we get the following commutative diagram
$$\xymatrix{
K_{x_B} \ar[d]_a \ar[r]^{k_{x_B}} &X \ar[d]_{x_B} \ar@/^/[ddr]^r\\
W_B \ar[r] \ar@/_/[drr]_b &B \ar@{.>}[dr]\\
&&Q.
}
$$
Since $\underline {x_B}$ is epimorphic in $\underline \h$ by Proposition \ref{eq2}, the above diagram implies that $\underline {x_B}$ is the cokernel of $\underline {k_{x_B}}$.
\end{proof}

By Theorem \ref{5.1}, in the case of the above theorem, the heart $\underline \h$ also becomes integral. Then by \cite[{Theorem 2}]{R}, $\underline \h$ is equivalent to a torsionfree class of a hereditary torsion theory in an abelian category induced by $\underline \h$. For more details, one can see \cite[{\S 4}]{R}.

\section{Existence of enough projectives/injectives}

We call an object $P\in \underline \h$ (proper-)projective if for any epimorphism (resp. cokernel) $\alpha:X\rightarrow Y$ in $\underline \h$, there exists an exact sequence
$$\Hom_{\underline \h}(P,X)\xrightarrow{\Hom_{\underline \h}(P,\alpha)} \Hom_{\underline \h}(P,Y)\rightarrow 0.$$
An (proper-)injective object is defined dually.\\
$\underline \h$ is said to have enough projectives if for any object $X\in \underline \h$, there is a cokernel $\delta:P\rightarrow X$ such that $P$ is proper-projective. Having enough injectives is defined dually.\\

In this section we give sufficient conditions that the heart $\underline \h$ of a twin cotorsion pair has enough projectives and has enough injectives.

\begin{lem}\label{6.1}
If a twin cotorsion pair $(\s,\T),(\U,\V)$ satisfies $\U\subseteq \T$, then we have $\Omega \s \subseteq \h$.
\end{lem}

\begin{proof}
We first have $\mathcal P \subseteq \U = \W$, then by definition $\Omega \s \subseteq \B^-$. But we observe that $\U \subseteq \T$ implies $\B^+=\B$, hence $\Omega \s \subseteq \h$.
\end{proof}

\begin{prop}\label{6.3}
Let $(\s,\T),(\U,\V)$ be a twin cotorsion pair satisfying $\U\subseteq \T$, then any object in $\Omega \s$ is projective in $\underline \h$.
\end{prop}

\begin{proof}
Let $B$ and $C$ be any objects in $\h$ and let $p:\Omega S\rightarrow C$ be any morphism.\\
Let $\underline g:B\rightarrow C$ be a morphism which is epimorphic in $\underline \h$, by Lemma \ref{epi} we can assume that it admits a short exact sequence
$$\xymatrix{A \ar@{ >->}[r]^{f} &B \ar@{->>}[r]^{g} &C.}$$
Since $B\in \h$ admits a short exact sequence $B\rightarrowtail W^B\twoheadrightarrow S^B$, then according to Proposition \ref{7}, there exists a commutative diagram
$$\xymatrix{
A \ar@{ >->}[r]^f \ar@{=}[d] &B \ar@{->>}[r]^g \ar[d] &C \ar[d]^q\\
A \ar@{ >->}[r] &W^B \ar@{->>}[r]^r \ar@{->>}[d] &D \ar@{->>}[d]\\
&S^B \ar@{=}[r] &S^B.
}
$$
By Lemma \ref{P1}, we obtain $D\in \B^-=\h$. Since $\underline {qg}=0$ and $\underline g$ is epimorphic in $\underline \h$, we have $\underline q=0$. By definition $\Omega S$ admits a short exact sequence
$$\xymatrix{\Omega S \ar@{ >->}[r]^{a} &P \ar@{->>}[r] &S} \text{ }(P\in \mathcal P, S\in \s).$$
Since $\underline {qp}=0$, $qp$ factors through $\W$. As $\Ext^1_\B(\s,\T)=0$, $a$ is a left $\T$-approximation of $\Omega S$. Thus there exists a morphism $s:P\rightarrow D$ such that $qp=sa$. Since $P$ is projective, there exists a morphism $t:P\rightarrow W^B$ such that $s=rt$. Hence by the definition of pull-back, we get the following commutative diagram
$$\xymatrix{
\Omega S \ar@/^/[drr]^p \ar@{.>}[dr]^h \ar@/_/[ddr]_{ta}\\
&B \ar[r]_g \ar[d] &C \ar[d]^q\\
&W^B \ar[r]_r &D}
$$
which implies that $\Omega S$ is projective in $\underline \h$.
\end{proof}

\begin{prop}\label{6.4}
Let $(\s,\T),(\U,\V)$ be a twin cotorsion pair satisfying $\U\subseteq \T$, then any object $\B\in \h$ admits an epimorphism $\alpha:\Omega S\rightarrow B$ in $\underline \h$.
\end{prop}

\begin{proof}
Let $B$ be any object in $\h$, consider commutative diagram \eqref{F5}. By Proposition \ref{PO}, the left square is a push-out. Now it suffices to show $\underline {s_B}$ is epimorphic in $\underline \h$.\\
Let $c:B\rightarrow C$ be any morphism in $\h$ such that $\underline {cs_B}=0$, then $cs_B$ factors through $\W$. Since $p$ is a left $\T$-approximation of $\Omega S$, there exits a morphism $d:P_{S^B}\rightarrow C$ such that $cs_B=dp$. Thus by the definition of push-out we have a commutative diagram
$$\xymatrix{
\Omega S^B \ar[r]^p \ar[d]_{s_B} &P_{S^B} \ar[d] \ar@/^/[ddr]^d\\
B \ar[r] \ar@/_/[drr]_c &W^B \ar@{.>}[dr]\\
&&C}
$$
which implies $\underline c=0$. Hence $\underline {s_B}$ is epimorphic in $\underline \h$.
\end{proof}

Moreover, we have

\begin{prop}\label{6.12}
Let $(\s,\T),(\U,\V)$ be a twin cotorsion pair satisfying $\U\subseteq \T$, then an object $B$ is projective in $\underline \h$ implies that $B\in \underline {\Omega \s}$.
\end{prop}

\begin{proof}
Suppose $B$ is projective in $\underline \h$, consider the commutative diagram \eqref{F5}. By Proposition \ref{6.4}, $\underline {s_B}$ is epimorphic in $\underline \h$, thus $B$ is a direct summand of $\Omega S^B$ in $\underline \h$. Hence by Lemma \ref{5.0} $B$ lies in $\underline {\Omega \s}$.
\end{proof}

From the following proposition we can get that in the case $\U\subseteq \T$ when the projectives in $\underline \h$ is enough.

\begin{prop}\label{enoughp}
Let $(\s,\T),(\U,\V)$ be a twin cotorsion pair satisfying $\U\subseteq \T$, then $\underline \h$ has enough projectives if and only if any indecomposable object $B\in \h-\U$ admits a short exact sequence
$$B\rightarrowtail S^1\twoheadrightarrow S^2$$
where $S^1,S^2\in \s$.
\end{prop}

\begin{proof}
We prove the "if" part first.\\
Since an object $B\in \h$ isomorphic to an object $B'\in \h$ in $\underline \h$ such that $B'$ does not have any direct summand in $\U$, we can only consider the obejct $B\in \h$ not having any direct summand in $\U$. Thus by assumption, $B$ admits a short exact sequence
$$B\rightarrowtail S^1\twoheadrightarrow S^2$$
where $S^1,S^2\in \s$. As $S^2$ admits a short exact sequence
$$\xymatrix{\Omega S^2 \ar@{ >->}[r]^b &P_{S^2} \ar@{->>}[r] &S^2.}$$
We have the following commutative diagram
$$\xymatrix{
\Omega S^2 \ar@{ >->}[r]^b \ar[d]_a &P_{S^2} \ar@{->>}[r] \ar[d] &S^2 \ar@{=}[d]\\
B \ar@{ >->}[r] &S^1 \ar@{->>}[r] &S^2.}
$$
Then we get a short exact sequence
$$\xymatrix{\Omega S^2 \ar@{ >->}[r]^{\svecv{a}{-b}} &{B\oplus P_{S^2}} \ar@{->>}[r] &S^1}$$
by Proposition \ref{PO}. Since $B\oplus P_{S^2}$ admits a short exact sequence
$$V\rightarrowtail U\twoheadrightarrow B\oplus P_{S^2}$$
where $V\in\V$ and $U\in \U=\W$, we obtain the following commutative diagram by Proposition \ref{7}
$$\xymatrix{
V \ar@{ >->}[r] \ar@{=}[d] &Q \ar@{->>}[r]^c \ar@{ >->}[d]^d &\Omega S^2 \ar@{ >->}[d]^{\svecv{a}{-b}}\\
V \ar@{ >->}[r] &U \ar@{->>}[d] \ar@{->>}[r] &B\oplus P_{S^2} \ar@{->>}[d]\\
&S^1 \ar@{=}[r] &S^1.
}
$$
Thus $Q\in \B^-=\h$ and $\underline {ca}=0$. We claim that $\underline a$ is the cokernel of $\underline c$ in $\underline \h$.\\
If $r:\Omega S^2\rightarrow M$ is a morphism in $\h$ such that $rc$ factors through $\W$, then there exists $e:U\rightarrow M$ such that $cr=ed$, since $d$ is a left $\T$-approximation of $Q$. Hence by definition of push-out, we get the following commutative diagram
$$\xymatrix{
Q \ar[r]^c \ar[d] &\Omega S^2 \ar[d]_{\svecv{a}{-b}} \ar@/^/[ddr]^r\\
U \ar[r] \ar@/_/[drr]_e &B\oplus P_{S^2} \ar@{.>}[dr]\\
&&M
}
$$
which implies that $\underline r$ factors through $\underline a$. Since $\underline a$ is epimorphic in $\underline \h$ by Proposition \ref{6.4}, we get that $\underline a$ is the cokernel of $\underline c$.\\
Now we assume that $\underline \h$ has enough projectives.\\
By Proposition \ref{6.12}, all the projective objects in $\underline \h$ lie in $\underline {\Omega S}$. Let $B$ be any indecomposable object in $\h-\U$ and $\beta:\Omega S\rightarrow B$ be a cokernel in $\underline \h$. Then by Lemma \ref{4.1}, we get a short exact sequence
$$\xymatrix{\Omega S \ar@{ >->}[r]^f &B' \ar@{->>}[r] &S'}$$
where $B'\in \h$ and $B'\simeq B$ in $\underline \h$ and $S'\in \s$. Since $\Omega S$ admits a short exact sequence
$$\xymatrix{\Omega S \ar@{ >->}[r]^p &P_S \ar@{->>}[r] &S}$$
we take a push-out of $f$ and $p$, then we get the following commutative diagram
$$\xymatrix{
\Omega S \ar@{ >->}[r]^f \ar@{ >->}[d]_p &B' \ar@{ >->}[d] \ar@{->>}[r] &S' \ar@{=}[d]\\
P_S \ar@{ >->}[r] \ar@{->>}[d] &Q' \ar@{->>}[r] \ar@{->>}[d] &S'\\
S \ar@{=}[r] &S.
}
$$
From the second row we get $Q'\in \s$. Since $B$ is indecomposable, it is a direct summand of $B'$. Hence by Lemma \ref{5.0}, $B$ admits a short exact sequence
$$B\rightarrowtail Q'\twoheadrightarrow S''$$
where $S''\in \s$.
\end{proof}

By duality, we have
\begin{prop}\label{6.5}
Let $(\s,\T),(\U,\V)$ be a twin cotorsion pair satisfying $\T\subseteq \U$, then any object in $\underline \h$ is injective if and only if it lies in $\underline {\Omega^- \V}$.
\end{prop}

\begin{prop}\label{6.11}
Let $(\s,\T),(\U,\V)$ be a twin cotorsion pair satisfying $\T\subseteq \U$, then any object $\B\in \h$ admits a monomorphism $\beta:B\rightarrow \Omega^-V$ in $\underline \h$ where $\Omega^-V\in \underline {\Omega^- \V}$.
\end{prop}

\begin{prop}\label{enoughi}
Let $(\s,\T),(\U,\V)$ be a twin cotorsion pair satisfying $\T\subseteq \U$, then the heart has enough injectives if and only if any object $B\in \h-\T$ admits a short exact sequence
$$V_2\rightarrowtail V_1\twoheadrightarrow B$$
where $V_1,V_2\in \V$.
\end{prop}

\section{Localisation on the heart of a special twin cotorsion pair}

Let $(\s,\T),(\U,\V)$ be a twin cotorsion pair on $\B$ such that $\T=\U$, in this case we get $\B^+=\B^-=\B$ and $\W=\T$, hence $\underline \h=\B/\T$. According to Theorem \ref{5.1}, $\B/\T$ is integral. Moreover, By Proposition \ref{6.3} (resp. Proposition \ref{6.5}), we obtain that any object in $\Omega \s$ (resp. $\Omega^-\V$) is projective (resp. injective) in $\B/\T$.\\
Let $R$ be the class of regular morphisms in $\B/\T$, then by Theorem \cite[{p173}]{R}, the localisation $(\B/\T)_{R}$ (if it exists) is abelian.

Till the end of this section we assume that $\B$ is skeletally small and $k$-linear over a field $k$ and has a twin cotorsion pair $(\s,\T),(\T,\V)$. We denote that by Proposition \ref{sp} it is equivalent to assume that $\B$ has a cotorsion pair $(\s,\T)$ such that $\s\subseteq \T$ and $\T$ is contravariantly finite.

Let $\mathcal D$ be a category and $R'$ is a class of morphisms on $\mathcal D$. If $R'$ admits both a calculus of right fractions and a calculus of left fractions (for details, see \cite[{\S 4}]{BM}), then the Gabriel-Zisman localisation ${\mathcal D}_{R'}$ at $R'$ (if it exists) has a very nice description. The objects in ${\mathcal D}_{R'}$ are the same as the objects in $\mathcal D$. The morphism from $X$ to $Y$ are of the form
$$\xymatrix{X &A \ar[r]^f \ar[l]_r &Y}$$
denoted by $[r,f]$ where $r$ lies in $R'$.\\
The localisation functor from $\mathcal D$ to ${\mathcal D}_{R'}$ takes a morphism $f$ to $[id,f]$. We denote this image by $[f]$. For $r\in R'$, $[r,id]$ is the inverse of $[r]$. We denote it $x_r$. Thus, every morphism has the form $[r,f]=[f]x_r$.

By \cite[{Corollary 4.2}]{BM}, $R$ admits both a calculus of right fractions and a calculus of left fractions.

For a subcategory $\mathcal C\subseteq \B$, we denote by $[\mathcal C]$ the full subcategory of $(\B/\T)_{R}$ which has the same objects as $\mathcal C$.

\begin{lem}
We have $\Omega \s/\mathcal P= \underline {\Omega \s}\simeq [\Omega \s]$.
\end{lem}

\begin{proof}
We first show that a morphism $f:\Omega S\rightarrow B$ factors through $\mathcal P$ if and only if it factors through $\T$.\\
Since $\mathcal P\subseteq \U=\T$, we only need to show $f$ factors through $\T$ implies it factors through $\mathcal P$. Suppose $f$ factors through $\T$. By definition $\Omega S$ admits the following short exact sequence
$$\xymatrix{\Omega S \ar@{ >->}[r]^q &P_S \ar@{->>}[r] &S}$$
where $P_S\in \mathcal P$, $S\in \s$ and $B$ admits the following short exact sequence
$$\xymatrix{V_B \ar@{ >->}[r] &W_B \ar@{->>}[r]^{w_B} &B.}$$
As $w_B$ is a right $\U$-approximation of $B$, there exists a morphism $a:\Omega S\rightarrow W_B$ such that $f=w_Ba$. Since $q$ is a left $\T$-approximation of $\Omega S$, there exists a morphism $b:P\rightarrow W_B$ such that $bq=a$, hence $f=w_Bbq$. Thus by definition we have $\Omega \s/\mathcal P= \underline {\Omega \s}$.\\
Let $L:\underline {\Omega \s}\rightarrow [\Omega \s]$ be the location of the localisation functor from $\B/\T$ to $(\B/\T)_{R}$. We claim that it is an equivalence. Obviously it is dense, it is faithful by \cite[{Lemma 4.4}]{BM} and full by \cite[{Lemma 5.4}]{BM}.
\end{proof}


Denote by $\Mod {\mathcal C}$ the category of contravariant additive functors from a category ${\mathcal C}$ to $\mod k$ for any category $\mathcal C$. Let $\mod {\mathcal C}$ be the full subcategory of $\Mod{\mathcal C}$ consisting of objects $A$ admitting an exact sequence:
$${\Hom_{\mathcal C}(-,C_1)}\xrightarrow{\beta} {\Hom_{\mathcal C}(-,C_0)}\xrightarrow{\alpha}  A\rightarrow0$$
where $C_0,C_1\in \mathcal C$.

Since $\underline {\Omega \s}\simeq [\Omega \s]$, We have $\mod (\Omega \s/\mathcal P)\simeq \mod [\Omega \s]$.

We give the following proposition which is an analogue of \cite[Lemma 5.5]{BM} (for more details, see \cite[{\S 5}]{BM}).

\begin{prop}\label{10.1}
If $(\s,\T),(\T,\V)$ is a twin cotorsion pair on $\B$ which is skeletally small, and let $R$ denote the class of morphisms which are both monomorphic and epimorphic in $\B/\T$, then
\begin{itemize}
\item[(a)] The projectives in $(\B/\T)_{R}$ are exactly the objects in $\Omega \s$.

\item[(b)] The category $(\B/\T)_{R}$ has enough projectives.

\end{itemize}
\end{prop}

For convenience, for any objects $X,Y\in \B$, we denote $\Hom_{[\B]}(X,Y)$ by $[X,Y]$. For any morphism $f:X\rightarrow Y$, we denote $\Hom_{[\B]}(-,[\underline f])$ by $-\circ [\underline f]$ and $\Hom_{[\B]}([\underline f],-)$ by $[\underline f]\circ -$.

Now we can prove the following theorem.

\begin{thm}\label{local}
Let $\B$ be a skeletally small, Krull-Schimdt, $k$-linear exact category with enough projecitves and injectives, containing a twin cotorsion pair
$$(\s,\T),(\T,\V).$$
Let $R$ denote the class of morphisms which are both monomorphic and epimorphic in $\B/\T$ and $(\B/\T)_{R}$ denote the localisation of $\B/\T$ at $R$, then
$$(\B/\T)_{R} \simeq \mod (\Omega \s/\mathcal P).$$
\end{thm}

\begin{proof}
It suffices to show $(\B/\T)_{R} \simeq \mod [\Omega \s]$.\\
From any object $B\in (\B/\T)_{R}$, there is a projective presentation of $B$:
$$\Omega S_1\xrightarrow{[\underline {d_1}]} \Omega S_0 \xrightarrow{[\underline {d_0}]} B\rightarrow 0$$
Let $\Omega S$ be any object in $[\Omega \s]$, we get the following exact sequence:
$$[\Omega S,\Omega S_1]\xrightarrow{\Omega S\circ [\underline {d_1}]} [\Omega S,\Omega S_0] \xrightarrow{\Omega S\circ [\underline {d_0}]} [\Omega S,B] \rightarrow 0$$
which induces a exact sequence in $\mod[\Omega \s]$:
$$[-,\Omega S_1]\xrightarrow{- \circ [\underline {d_1}]} [-,\Omega S_0] \xrightarrow{- \circ[\underline {d_0}]} [-,B] \rightarrow 0$$
Now we can define a functor $\Phi:(\B/\T)_{R} \rightarrow \mod [\Omega \s]$ as follows:
\begin{align*}
B \mapsto [-,B],\\
[\underline f] \mapsto -\circ [\underline f].
\end{align*}
$\bullet$~Let us prove that $\Phi$ is faithful.\\
For any morphism $[\underline f]:B\rightarrow B'$  we have the following commutative diagram
$$\xymatrix{
\Omega S_1 \ar[r]^{[\underline {d_1}]} \ar@{.>}[d]_{[\underline {f_1}]} &\Omega S_0 \ar[r]^{[\underline {d_0}]} \ar@{.>}[d]^{[\underline {f_0}]} &B \ar[r] \ar[d]^{[\underline f]} &0\\
\Omega S_1' \ar[r]_{[\underline {d_1'}]} &\Omega S_0 \ar[r]_{[\underline {d_0'}]} &B \ar[r] &0
}
$$
in $(\B/\T)_{R}$ which induces a commutative diagram in $\mod [\Omega \s]$
$$\xymatrix{
[-,\Omega S_1] \ar[r]^{- \circ [\underline {d_1}]} \ar@{.>}[d]_{-\circ[\underline {f_1}]} &[-,\Omega S_0] \ar[r]^{- \circ[\underline {d_0}]} \ar@{.>}[d]^{-\circ[\underline {f_0}]} &[-,B] \ar[r] \ar[d]^{-\circ [\underline f]} &0\\
[-,\Omega S_1'] \ar[r]_{- \circ [\underline {d_1'}]} &[-,\Omega S_0'] \ar[r]_{- \circ[\underline {d_0'}]} &[-,B'] \ar[r] &0.
}
$$
Hence if $-\circ [\underline f]=0$, we obtain $-\circ [\underline {d_0'f_0}]=0$, which implies $[\underline {d_0'f_0}]=0$. Thus $[\underline f]=0$.\\
$\bullet$~Let us prove that $\Phi$ is full.\\
For any morphism $\alpha:[-,B]\rightarrow[-,B']$, we have the following commutative diagram
$$\xymatrix{
[-,\Omega S_1] \ar[r]^{- \circ [\underline {d_1}]} \ar[d]_{\alpha_1} &[-,\Omega S_0] \ar[r]^{- \circ[\underline {d_0}]} \ar[d]_{\alpha_0} &[-,B] \ar[r] \ar[d]^{\alpha} &0\\
[-,\Omega S_1'] \ar[r]_{- \circ [\underline {d_1'}]} &[-,\Omega S_0'] \ar[r]_{- \circ[\underline {d_0'}]} &[-,B'] \ar[r] &0.
}
$$
in $\mod [\Omega \s]$. By Yoneda's Lemma, there exists $[\underline {f_i}]:\Omega S_i\rightarrow \Omega S_i'$ such that $\alpha_i=- \circ[\underline {f_i}]$. Hence there is a commutative diagram
$$\xymatrix{
[-,\Omega S_1] \ar[r]^{- \circ [\underline {d_1}]} \ar[d]_{-\circ[\underline {f_1}]} &[-,\Omega S_0] \ar[r]^{- \circ[\underline {d_0}]} \ar[d]^{-\circ[\underline {f_0}]} &[-,B] \ar[r] \ar@{.>}[d]^{-\circ [\underline f]} &0\\
[-,\Omega S_1'] \ar[r]_{- \circ [\underline {d_1'}]} &[-,\Omega S_0'] \ar[r]_{- \circ[\underline {d_0'}]} &[-,B'] \ar[r] &0.
}
$$
in $(\B/\T)_{R}$, thus $\alpha=-\circ [\underline f]$.\\
$\bullet$~Let us prove that $\Phi$ is dense:\\
We first show that $\mod [\Omega \s]$ is abelian. It is enough to show that $[\Omega \s]$ has pseudokernels. Let $\alpha:\Omega S_1\rightarrow \Omega S_0$ be a morphism in $[\Omega \s]$, then since $(\B/\T)_{R}$ is abelian, there exists a kernel $\beta:K\rightarrow \Omega S_1$ in $(\B/\T)_{R}$. By Proposition \ref{10.1}, there exists a epimorphism $\gamma: \Omega S\rightarrow K$. We observe that $\beta\gamma$ is a pseudokernel of $\alpha$.\\
 Let $F\in \mod [\Omega \s]$ which admits an exact sequence
$$[-,\Omega S_1]\xrightarrow{- \circ \gamma} [-,\Omega S_0] \rightarrow F \rightarrow 0$$
where $\gamma \in [\Omega S_1,\Omega S_0]$. Let $B=\Coker\gamma$
then we get an exact sequence
$$\Omega S_1\xrightarrow{\gamma} \Omega S_0 \rightarrow B\rightarrow 0$$
in $(\B/\T)_{R}$. Hence $F\simeq [-,B]$.
\end{proof}

\section{Examples}

In this section we give several examples of twin cotorsion pair, and we also give some view of the relation between the heart of a cotorsion pair and the hearts of its two components.\\
First we introduce some notations. Let $\mathcal C$ be a subcategory of $\B$, we set
\begin{itemize}
\item[(a)] ${\mathcal C}^{\bot_n}=\{X\in \B \text{ }| \text{ } \Ext^i_\B(\mathcal C,X)=0, \text{ } 0<i\leq n \}$.
\item[(b)] $^{\bot_n}{\mathcal C}=\{X\in \B \text{ }| \text{ } \Ext^i_\B(X,\mathcal C)=0, \text{ } 0<i\leq n \}$.
\item[(c)] ${\mathcal C}^{\bot}=\{X\in \B \text{ }| \text{ } \Ext^i_\B(\mathcal C,X)=0, \text{ } \forall i>0 \}$.
\item[(d)] $^{\bot}{\mathcal C}=\{X\in \B \text{ }| \text{ } \Ext^i_\B(X,\mathcal C)=0, \text{ } \forall i>0 \}$.
\end{itemize}

According to \cite[\S 7.2]{HO}, we give the following definition.

\begin{defn}\label{h}
A cotorsion pair $(\U,\V)$ is called a hereditary cotorsion pair if $\Ext^i_\B(\U,\V)=0,i>0$.
\end{defn}

The following proposition can be easily checked by definition.

\begin{prop}
For a cotorsion pair $(\U,\V)$, the following conditions are equivalent.
\begin{itemize}
\item[(a)] $(\U,\V)$ is hereditary.
\item[(b)] $\V=\U^\bot$.
\item[(c)] $\U={^\bot\V}$.
\item[(d)] $\Omega \U\subseteq \U$.
\item[(e)] $\Omega^- \V\subseteq \V$.
\end{itemize}
\end{prop}


\begin{rem}
We can call a pair of subcategories $(\U,\V)$ a \emph{co-t-structure} on $\B$ if it is a hereditary cotorsion pair, since by the proposition above the hereditary cotorsion pair on $\B$ is just an analogue of the co-t-structure on triangulated category.
\end{rem}

\begin{exm}
We introduce two trivial hereditary cotorsion pairs:
$$(\mathcal P,\B)\text{ and }(\B,\mathcal I).$$
We observe that in these two cases the hearts are $0$. These two cotorsion pairs also form a twin cotorsion pair
$$(\mathcal P,\B),(\B,\mathcal I).$$
We observe that its heart is also $0$.
\end{exm}

\begin{exm}\label{7.2}
Let $\Lambda$ be an artin algebra and $T$ be a cotilting module of finite injective dimension, denote
$$\mathcal X:={^\bot T} \text{ and } \mathcal Y:=(^\bot T)^\bot.$$
By \cite[{Theorem 5.4, Corollary 5.10, Proposition 3.3.}]{A}, $(\mathcal X, \mathcal Y)$ is a hereditary cotorsion pair.
By \cite[{Proposition 3.3, (c, iii)}]{A}, we get
$$\W \subseteq(\mod \Lambda)^+\subseteq \mathcal Y.$$
Dually, by \cite[{Proposition 3.3, (d, iii)}]{A}, we get
$$\W \subseteq(\mod \Lambda)^-\subseteq \mathcal X.$$
Then $\h= (\mod \Lambda)^+ \cap (\mod \Lambda)^-\subseteq \mathcal X \cap \mathcal Y=\W$, hence $\underline \h=0$.\\
By \cite[{Proposition 1.8}]{A}, $({^{\bot_1}T},(^{\bot_1}T)^{\bot_1})$ is a cotorsion pair. According to \cite[{\S 2}]{A}, ${^{\bot_1}T},(^{\bot_1}T)^{\bot_1}$ is also a a cotorsion pair. Hence by definition
$$(^\bot T,(^\bot T)^\bot), ({^{\bot_1}T},(^{\bot_1}T)^{\bot_1})$$
form a twin cotorsion pair. We can also observe that its heart is trivial.
\end{exm}

In fact, we have
\begin{prop}\label{7.3}
If one cotorsion pair in a twin cotorsion pair $(\s,\T),(\U,\V)$ is hereditary, then this twin cotorsion pair has a trivial heart, \emph{i.e.} its heart is zero.
\end{prop}

\begin{proof}
We prove that if $(\s,\T)$ is hereditary, then $\W=\V\cap \s=\B^+\cap \B^-$, another part is by dual.\\
For any object $B\in\B^-$, there is a short exact category
$$B\rightarrowtail W^B\twoheadrightarrow S^B.$$
Since we have the following exact sequence
$$0=\Ext^1_\B(W^B,\T) \rightarrow \Ext^1_\B(B,\T) \rightarrow \Ext^2_\B(S^B,\T)=0$$
which implies $B\in\s$. Hence $\B^-=\s$. Dually, $\B^+=\V$. Hence $\W\subseteq\B^+\cap \B^-=\V\cap \s\subseteq \W$, this implies $\underline \h=0$.
\end{proof}

Recall that $\M$ is \emph{$n$-cluster tilting} if it satisfies the following conditions
\begin{itemize}
\item[(a)] $\M$ is contravariantly finite and covariantly finite in $\B$,
\item[(b)] $\M^{\bot_n}=\M$.
\item[(c)] ${^{\bot_n}\M}=\M$.
\end{itemize}
A $2$-cluster tilting subcategory is usually called \emph{cluster tilting} subcategory.\\
Let $\M$ be a cluster tilting subcategory of $\B$. Remark that $\mathcal P\subseteq \mathcal M$ and $\mathcal I\subseteq \mathcal M$. For each object $B\in\B$, we have two short exact sequences
\begin{align*}
\xymatrix{B \ar@{ >->}[r]^{f} &M \ar@{->>}[r] &N,}\\
\xymatrix{N' \ar@{ >->}[r] &M' \ar@{->>}[r]^{g} &B}
\end{align*}
that $f$ (resp. $g$) is a left (resp. right) $\M$-approximation of $B$. We observe $N\in {^{\bot_1}\M}=\M$ (resp. $N'\in {\M}^{\bot_1}=\M$), therefore $(\M,\M)$ is a cotorsion pair. In this case, $\W=\M$ and $\B^+=\B^-=\B$, thus $\underline \h=\uB=\B/\M$, which is abelian also by \cite{DL}.\\
Moreover, any object in $\Omega \M$ (resp. $\Omega^- \M$) is projective (resp. injective) in $\B/\M$, and by Proposition \ref{enoughp},\ref{enoughi}, $\B/\M$ has enough projectives and enough injectives.

\begin{prop}
A subcategory $\M$ in $\B$ is cluster tilting if and only if $(\M,\M)$ is a cotorsion pair on $\B$.
\end{prop}

\begin{proof}
From the above discussion, we know that $(\M,\M)$ is a cotorsion pair if $\M$ is cluster tilting, so it remains to show the "only if" part. But it is just followed by the definition of cotorsion pair and Lemma \ref{3}.
\end{proof}

In the following examples, we denote by "$\circ$" in a quiver the objects belong to a subcategory and by "$\cdot$" the objects do not.

\begin{exm}\label{7.6}
Let $\Lambda$ be the path algebra of the following quiver
$$\xymatrix{1 &2 \ar[l] &3 \ar[l] &4 \ar@/_10pt/@{.>}[lll] \ar[l]}$$
then we obtain the $AR$-quiver $\Gamma(\mod \Lambda)$ of $\mod \Lambda$.
$$\xymatrix@C=0.4cm@R0.4cm{
{\begin{smallmatrix}
1
\end{smallmatrix}} \ar@{.}[rr] \ar[dr]
&&{\begin{smallmatrix}
2
\end{smallmatrix}} \ar@{.}[rr] \ar[dr]
&&{\begin{smallmatrix}
3
\end{smallmatrix}} \ar@{.}[rr] \ar[dr]
&&{\begin{smallmatrix}
4
\end{smallmatrix}}\\
&{\begin{smallmatrix}
2&&\\
&1&
\end{smallmatrix}} \ar@{.}[rr] \ar[dr] \ar[ur]
&&{\begin{smallmatrix}
3&&\\
&2&
\end{smallmatrix}} \ar@{.}[rr] \ar[dr] \ar[ur]
&&{\begin{smallmatrix}
4&&\\
&3&
\end{smallmatrix}} \ar[ur]\\
&&{\begin{smallmatrix}
3&&\\
&2&\\
&&1
\end{smallmatrix}} \ar[ur]
&&{\begin{smallmatrix}
4&&\\
&3&\\
&&2
\end{smallmatrix}} \ar[ur]
}
$$
Let $\M=\{X\in\mod \Lambda \text{ } | \text{ } \Ext^1_\B(X,\Lambda)=0 \}$, then by \cite[{Proposition 1.10, 1.9}]{A}, $(\M,\M^{\bot_1})$ is a cotorsion pair on $\mod \Lambda$. But
$$\M=\xymatrix@C=0.1cm@R0.1cm{
\circ  &&\cdot  &&\cdot  &&\circ\\
&\circ  &&\cdot  &&\circ \\
&&\circ  &&\circ
}
$$
which consisting of all the direct sums of indecomposable projectives and indecomposable injectives. We observe that in fact $\M=\M^{\bot_1}$ and hence it is a cluster tilting subcategory. And the quiver of the quotient category $\mod \Lambda/\M$ is
$$\xymatrix{
{\begin{smallmatrix}
2
\end{smallmatrix}} \ar@{.}[rr] \ar[dr]
&&{\begin{smallmatrix}
3
\end{smallmatrix}}\\
&{\begin{smallmatrix}
3&&\\
&2&
\end{smallmatrix}} \ar[ur]
}
$$
which is equivalent to the $AR$-quiver of $A_2$.
\end{exm}

\begin{exm}\label{ex1}
Take the notion of the former example, Let
$$\M'=\xymatrix@C=0.1cm@R0.1cm{
\circ &&\cdot  &&\cdot &&\cdot\\
&\circ &&\cdot  &&\circ\\
&&\circ &&\circ
}
$$
then by \cite[{Proposition 1.10, 1.9}]{A}, $(\M',{\M'}^{\bot_1})$ is a cotorsion pair and
$${\M'}^{\bot_1}=\xymatrix@C=0.1cm@R0.1cm{
\circ &&\cdot &&\circ &&\circ\\
&\circ &&\cdot &&\circ \\
&&\circ &&\circ
}
$$
hence it contains $\Lambda$. Obviously it is closed under extension and contravariantly finite, then by \cite[{Proposition 1.10, 1.9}]{A}, $({\M'}^{\bot_1},({\M'}^{\bot_1})^{\bot_1})$ is also a cotorsion pair on $\mod \Lambda$ and
$$({\M'}^{\bot_1})^{\bot_1}=\xymatrix@C=0.1cm@R0.1cm{
\circ &&\cdot &&\cdot &&\circ\\
&\cdot &&\cdot &&\circ\\
&&\circ &&\circ
}
$$
Thus we get a twin cotorsion pair
$$(\M',{\M'}^{\bot_1}),({\M'}^{\bot_1},({\M'}^{\bot_1})^{\bot_1}).$$
Then the quiver of $\mod \Lambda/{\M'}^{\bot_1}$ is ${\begin{smallmatrix}
2
\end{smallmatrix}} \rightarrow {\begin{smallmatrix}
3&&\\
&2&.
\end{smallmatrix}}$
The quiver of quotient category $\Omega \M'/\mathcal P$ is just ${\begin{smallmatrix}
2.
\end{smallmatrix}}$
Hence we get $(\mod \Lambda/{\M'}^{\bot_1})_{R}\simeq \mod (\Omega \M'/\mathcal P)$.
\end{exm}

From Example \ref{ex1}, we see that there exist two cotorsion pairs which have non-trivial hearts form a twin cotorsion pair also having a non-trivial heart. From the following example, we see that even two components of a twin cotorsion pair have non-trivial hearts, the heart of the twin cotorsion pair itself can be zero.

\begin{exm}
Let $\Lambda$ be the $k$-algebra given by the quiver
$$
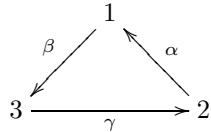
\xymatrix{
&1 \ar[dl]_{\beta}\\
3 \ar[rr]_{\gamma} &&2 \ar[ul]_{\alpha}
}
$$
and bound by $\alpha\beta=0$ and $\beta\gamma\alpha=0$. Then its AR-quiver $\Gamma(\mod \Lambda)$ is given by
$$
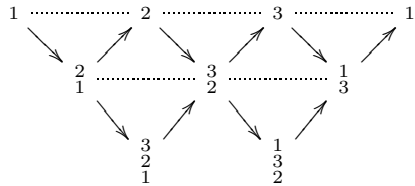
\xymatrix@C=0.4cm@R0.4cm{
{\begin{smallmatrix}
1
\end{smallmatrix}} \ar[dr] \ar@{.}[rr]
&&{\begin{smallmatrix}
2
\end{smallmatrix}} \ar[dr] \ar@{.}[rr]
&&{\begin{smallmatrix}
3
\end{smallmatrix}} \ar[dr] \ar@{.}[rr]
&&{\begin{smallmatrix}
1
\end{smallmatrix}}\\
&{\begin{smallmatrix}
2\\
1
\end{smallmatrix}} \ar[ur] \ar[dr] \ar@{.}[rr]
&&{\begin{smallmatrix}
3\\
2
\end{smallmatrix}} \ar[ur] \ar[dr] \ar@{.}[rr]
&&{\begin{smallmatrix}
1\\
3
\end{smallmatrix}} \ar[ur]\\
&&{\begin{smallmatrix}
3\\
2\\
1
\end{smallmatrix}} \ar[ur]
&&{\begin{smallmatrix}
1\\
3\\
2
\end{smallmatrix}} \ar[ur]
}
$$
Here, the first and the last columns are identified. Let
$$\s=\xymatrix@C=0.1cm@R0.1cm{
\cdot &&\circ &&\cdot  &&\cdot\\
&\circ &&\cdot &&\circ \\
&&\circ  &&\circ \\
} \quad
\T=\xymatrix@C=0.1cm@R0.1cm{
\cdot  &&\cdot &&\circ &&\cdot\\
&\circ  &&\cdot &&\circ \\
&&\circ  &&\circ \\
}
$$
and
$$\U=\xymatrix@C=0.1cm@R0.1cm{
\circ  &&\circ  &&\cdot  &&\circ\\
&\circ  &&\cdot &&\circ\\
&&\circ  &&\circ \\
} \quad
\V=\xymatrix@C=0.1cm@R0.1cm{
\cdot  &&\cdot  &&\cdot &&\cdot\\
&\circ  &&\cdot  &&\circ \\
&&\circ &&\circ \\
}
$$
The heart of cotorsion pair $(\s,\T)$ is $\add({\begin{smallmatrix}
1
\end{smallmatrix}})$ and the heart of cotorsion pair $(\U,\V)$ is $\add({\begin{smallmatrix}
3
\end{smallmatrix}})$. But when we consider the twin cotorsion pair $(\s,\T),(\U,\V)$, we get $\W=\V$ and
$$(\mod \Lambda)^-/\W=\add({\begin{smallmatrix}
1
\end{smallmatrix}}\oplus
{\begin{smallmatrix}
2
\end{smallmatrix}})
 \text{ and } (\mod \Lambda)^+/\W=\add({\begin{smallmatrix}
3
\end{smallmatrix}})$$
hence its heart is zero.
\end{exm}

\section*{Acknowledgments}
The author would like to thank Laurent Demonet and Osamu Iyama for their helpful advices and corrections.

\end{document}